\documentclass[reqno]{amsart}

\usepackage{style}

\author{Mireille Boutin} \address{Department of Mathematics and
  Computer Science, Eindhoven University of Technology, 5600 MB
  Eindhoven, The Netherlands}
\email{m.boutin@tue.nl}%
\author{Gregor Kemper} \address{Technical University of Munich,
  Germany; TUM School of Computation, Information and Technology,
  Department of Mathematics} \email{kemper@tum.de}

\title{Localization from Pseudoranges: Quadrics and Duality}

\date{May 27, 2026}

\begin{document}

\begin{abstract}
  This paper gives a complete description of the solutions of the
  global positioning problem, emphasizing the under-determined case. We
  show that the solutions form a quadric, which may degenerate in
  various ways. Perhaps more surprisingly, the satellite positions
  also lie on a quadric, and these two quadrics exhibit a remarkable
  duality: They live on perpendicular affine spaces but share the same
  axis of symmetry. Moreover, the vertices of one quadric are the foci
  of the other and vice versa.

  The results of this paper are not only applicable to the global
  positioning problem, but to a wider class of problems known as pseudorange-multilateration. This includes a range of real-world localization problems
  where a signal is emitted at an unknown emission
  time, and received by sensors at known positions. In particular, the paper can be useful for solving an under-determined multilateration problem in the presence of additional constraints. We illustrate this with two examples: locating a cleaning robot on the ground and locating a raft on the ocean.
\end{abstract}

\maketitle

\newcommand{\ve}[1]{\mathbf{#1}}%
\newcommand{\vh}[1]{\widehat{\ve{#1}}}%
\newcommand{\Sol}{Q_\mathrm{sol}}
\newcommand{\Sat}{Q_\mathrm{sat}}
\newcommand{\Asol}{\mathcal A_\mathrm{sol}}
\newcommand{\Asat}{\mathcal A_\mathrm{sat}}

\section*{Introduction} \label{sIntro}%
Geolocation is the process of determining the physical position of a user or an event~[\citenumber{gentile2012geolocation}]. This task arises in a wide range of applications. In logistics~[\citenumber{shamsuzzoha2011real},\citenumber{shamsuzzoha2013performance},\citenumber{dong2023design}] and warehousing~[\citenumber{LeeWarehouse2019}], for instance, tracking the location of moving objects is critical. In everyday life, most of us rely on our smartphones’ ability to pinpoint our location as we navigate and travel.

Geolocation can be performed using landmarks situated at fixed, precisely known locations. The distance to one of these landmarks is  called the ``range." Geolocation from range values is obtained by finding the intersection of the circles or spheres defined by the ranges, a process called trilateration~[\citenumber{arlinghaus1994practical}]. Alternatively, one can use the angles of the lines of sight to the landmarks and determine ones location by triangulation.

In 2D space, geolocation using range values requires at least three landmarks. This is because three circles generally intersect at a single point. In 3D space, that number increases to four because, in general, three spheres intersect at two points and so a fourth circle is required to disambiguate the position.
Thus, in outdoor areas with high cellular coverage, one can determine their location using four (generically positioned) cell towers. Similarly, in indoor environments, where cell phone reception is poorer~[\citenumber{pahlavan2000overview}], one can instead geolocate in 3D space by using four (well-positioned) wifi signal emitters~[\citenumber{cheng2005accuracywifi},\citenumber{retscher2007performance}]. 
In some cases, the distances to the landmarks are only known up to an unknown (constant) shift. In that case, one talks about geolocation from ``pseudoranges" and the underlying localization technique is called ``multilateration."

One example of this is the problem is gunshot localization, where a 
gun, fired at an unknown location and time, 
is heard at known times at several positions nearby using an electro-acoustic system~[\citenumber{gunshothistory2013}].  By dividing the reception times of the gun sound by the speed of sound, one obtains the 
 pseudoranges to  the location from which the gun was fired. Multilateration can then be used to try to pinpoint that location. Other sound events (e.g., door clapping, car backfiring, fireworks) can also be located in a similar fashion.
 
 Another well-known example is global positioning. In that case, a user receives the position and emission times of several precisely synchronized emitters. Outdoors, these could be satellites such as those from GPS or Galileo. Under water, these could be seafloor mounted emitters~[\citenumber{UAVacoustic1998}] such as those used in long baselines (LBL) acoustic positioning systems~[\citenumber{BATISTAlongbaseline2015},\citenumber{ZhangdesignLBL2019}].
 The clocks of the emitters are precisely synchronized with each other, but not with the user's clock. Thus the difference between the emission times of the signals and the reception time on the user's clock only reveals the distance to the emitters up to an unknown constant. As the constant is the same for all the emitters, this determines the pseudoranges to each of the emitters.  

In general, geolocation requires the landmarks' locations to be generic, that is to say, one needs to avoid certain ``bad" configurations from which a unique location cannot be determined. Another situation to avoid is when the landmarks are close to being in a bad configuration, as this can create numerical issues that can prevent accurate localization when the data is perturbed by noise. For example, in global positioning using satellites, one should avoid satellite configurations that are visually close in the sky. This phenomenon is described by the so-called ``geometric dilution of precision" [\citenumber{langley1999dilution}].
Other positions to avoid are the individual sheets of hyperboloids of revolution with one focus at the user position~[\citenumber{Boutin:Kemper:2024}]. As shown in~[\citenumber{Boutin:Kemper:Notices}], localization becomes numerically unstable as the emitters position gets closer to such a configuration. 

Another issue that may arise is that fewer landmarks than needed are available, what we call ``degenerated scenarios." 
For example, 
in order to obtain a signal from a landmark using electromagnetic signals, it is necessary to have a clear line of sight to the landmark. This can be difficult in urban environments where high buildings and narrow streets can obstruct the signal~[\citenumber{chen2022phone}] creating so-called {\em urban canyon}.  Such signal blockage also occurs in natural settings, where the view can be obstructed by trees, hills, or large rocks. In some scenarios, an adversary may be purposefully jamming the satellite signals to prevent location. 
Since a minimal number of landmark signal observations is necessary for unique location, there are many situations where geolocation is not possible.

One option to disambiguate the location when two few landmarks are in view is to use other signals available. For example,
workaround solutions have been proposed for global positioning in urban environment, such as using the phase of the signals~[\citenumber{mohamadi2025phase}] or using other 
telecommunications networks~[\citenumber{koelemeij2022hybrid}].
Another option is to use known constraints on the user position. 
For example,
assuming omnidirectional antennas, the minimum number of cell towers required to determine the 3D location of a user is four. But if the user is assumed to be on the earth surface, then this can be done with only three towers. In particular, if the user is assumed to be on the earth sphere, then the center of the earth can be used as a fourth landmark, and the earth sphere used as a fourth range. This is a reasonable assumption in many scenarios, e.g., when traveling at sea level (on a ship.) Assuming that the sphere centers do not lie on a straight line,
intersecting the four spheres then yields a single solution.
Similar, for global positioning, four is the minimum number of satellites required. With three satellites, the solution set is one-dimensional. But if the user is assumed to be on the earth sphere, three satellites yield a finite number of location possibility-- at most four~[\citenumber{Boutin:Kemper:Eggermont:2025}].

In the following, we focus on the problem of localization from pseudoranges 
(e.g., global positioning, gunshot location, etc.) and analyze the degenerated 
scenarios in detail: that is, the case where too few landmarks for precise 
location are available. We show that the set of possible locations (i.e., the 
solution set) determined by such problems form a quadric, which may degenerate 
in various ways. We also show that the positions of satellites (in the global 
positioning situation) or sensors (in the situation of gunshot location etc.) 
corresponding to the solution set form a quadric as well, and these two 
quadrics exhibit a remarkable duality: they live on perpendicular affine spaces but share the same axis of symmetry. Moreover, the vertices of one quadric are the foci of the other and vice versa. We invite readers to take a look at \cref{fQuadrics} already now, where the situation is illustrated.

This duality has practical consequences. For example, if there are three 
sensors lying on a common plane, such as the ground plane, inside $\RR^3$, our 
results imply that the quadric of solutions is a conic section that passes 
through this plane perpendicularly. This is important since if it is known that 
the actual solution (e.g. the gunshot location) lies very close to that plane, 
then it can be determined very accurately because of the perpendicularity.

Among other things, our results provide a clear description of the solution set of any degenerate localization from pseudorange problem, and hold for any space dimension. Thus, they can be used as a starting point to disambiguate the location, be it by using other signals (e.g., wifi or cellular towers) or by using  a priori knowledge about the location such as geometric navigation constraints, as already mentioned.
We illustrate the latter using two scenarios. In the first scenario, presented in \cref{sec:vacuum_robot}, an indoor robot vacuum uses emitters affixed to the ceiling to determine its location on the ground.  In the second scenario, presented in \cref{sec:life_boat}, we discuss the problem of locating a raft on the ocean.

But before coming to applications, the article works towards the more theoretical results mentioned above. This happens in \cref{sSpheres,sQsol,sQsat,sShape,sIneq,sRank,sDuality}, with \cref{sQsat} containing the main technical work and \cref{sShape} containing the main result, \cref{tSolSat}. Throughout these more theoretical sections we will be using the language of the global positioning problem, so in particular we will say ``satellite positions'' for what in other situations are the known positions of sensors. Consequently, the two quadrics we find will be called the quadric of solutions and the quadric of satellites.

\section{Prelude: Global positioning and sphere
  intersection} \label{sSpheres}%

The well-known \df{global positioning equations} are
\begin{equation} \label{eqGPS}%
  \lVert\ve s_i - \ve x\rVert = t_i - b \qquad (i = 1 \upto m),
\end{equation}
where $\lVert\cdot\rVert$ denotes the Euclidean norm.
Here~$\ve s_i \in \RR^n$ (with $n \ge 1$) are the known positions of
the satellites in view, $m \ge 1$ is their number, the~$t_i \in \RR$
are the known pseudoranges, $\ve x \in \RR^n$ is the unknown user
position, and~$b \in \RR$ is the unknown bias of the user clock. By
squaring both sides we get a slightly weaker system
\begin{equation} \label{eqGPSsquared}%
  \lVert\ve s_i - \ve x\rVert^2 = (t_i - b)^2 \qquad (i = 1 \upto m).
\end{equation}
More precisely,~\cref{eqGPS} is equivalent to~\cref{eqGPSsquared}
together with the inequalities
\begin{equation} \label{eqIneq}%
  t_i \ge b \qquad (i = 1 \upto m).
\end{equation}
We will work with the weaker equations~\cref{eqGPSsquared} through
most of this article, and then go back to the original
equations~\cref{eqGPS} in \cref{sIneq}, where~\cref{eqIneq} will be
considered. Let us write
\begin{equation} \label{eqX}%
  \mathfrak X := \bigl\{(b,\ve x) \in \RR \times \RR^n\ \bigl| \
  \lVert\ve s_i - \ve x\rVert^2 = (t_i - b)^2 \ \text{for} \ i = 1
  \upto m\bigr\} \subseteq \RR \times \RR^n
\end{equation}
for the set of solutions of~\cref{eqGPSsquared}.

The system~\cref{eqGPSsquared} is a special case of the \df{sphere
  intersection problem}: Let $\langle\cdot,\cdot\rangle_Q$ be a
nondegenerate symmetric bilinear form on $\RR^r$ given by a symmetric
matrix $Q \in \RR^{r \times r}$. Given points
$\vh s_1 \upto \vh s_m \in \RR^r$ and numbers $d_1 \upto d_m \in \RR$,
solve the system
\begin{equation} \label{eqSpheres}%
  \langle\vh s_i - \vh x,\vh s_i - \vh x\rangle_Q = d_i \qquad (i = 1
  \upto m)
\end{equation}
for~$\vh x \in \RR^r$. So~\cref{eqGPSsquared} is the special case
$r = n+1$, $Q = \diag(-1,1 \upto 1)$,
$\vh s_i := \left(\begin{smallmatrix} t_i \\ \ve
    s_i\end{smallmatrix}\right)$,
$\vh x := \left(\begin{smallmatrix} b \\ \ve
    x\end{smallmatrix}\right)$, and $d_1 = \cdots = d_m = 0$.

Let us recall a solution procedure for the sphere intersection
problem~\cref{eqSpheres}, which was also presented
in~[\citenumber{Boutin:Kemper:Notices}]. With
\begin{equation} \label{eqA}%
  A :=
  \left(
    \begin{array}{c|c}
      2 \vh s_1^T Q & -1 \\
      \vdots & \vdots \\
      2 \vh s_m^T Q & -1
    \end{array}
  \right) \in \RR^{m \times (r+1)},  
\end{equation}
the system~\cref{eqSpheres} can be rewritten as the matrix equation
\begin{equation} \label{eqMatrix}%
  A \cdot
  \begin{pmatrix}
    \vh x \\
    \langle\vh x,\vh x\rangle_Q
  \end{pmatrix}
  =
  \begin{pmatrix}
    \langle\vh s_1,\vh s_1\rangle_Q - d_1 \\
    \vdots \\
    \langle\vh s_m,\vh s_m\rangle_Q - d_m
  \end{pmatrix}.
\end{equation}
Introducing a further unknown~$\lambda \in \RR$, we see that this is
equivalent to the equations
\begin{equation} \label{eqLinear}%
  A \cdot
  \begin{pmatrix}
    \vh x \\
    \lambda
  \end{pmatrix}
  =
  \begin{pmatrix}
    \langle\vh s_1,\vh s_1\rangle_Q - d_1 \\
    \vdots \\
    \langle\vh s_m,\vh s_m\rangle_Q - d_m
  \end{pmatrix}
\end{equation}
and
\begin{equation} \label{eqQuadratic}%
  \langle\vh x,\vh x\rangle_Q - \lambda = 0.
\end{equation}

So if we write $L$ for the set of solutions of~\cref{eqLinear}, which
is an affine subspace of $\RR^{r+1}$, then~\cref{eqSpheres} can be
solved by imposing the equation~\cref{eqQuadratic} on the points of
$L$. So as a special case we obtain a solution procedure for the
squared global positioning equations~\cref{eqGPSsquared}.

\section{The quadric of solutions} \label{sQsol}

In this section we analyze the solution of the squared global
positioning equations~\cref{eqGPSsquared}, obtained in \cref{sSpheres}
as a special case of the sphere intersection problem, in more
detail. We start by assuming that the matrix
\begin{equation} \label{eqB}%
  B :=
  \begin{pmatrix}
    2 \ve s_1^T & -1 \\
    \vdots & \vdots \\
    2 \ve s_m^T & -1
  \end{pmatrix} \in \RR^{m \times (n+1)},
\end{equation} 
has rank~$m$. (This assumption will be made during
\cref{sQsol,sQsat,sShape,sIneq}). This is equivalent to each of the
following assumptions:
\begin{itemize}
\item the linear space
  $U := \langle\ve s_i - \ve s_1 \mid 2 \le i \le
  m\rangle_\mathrm{lin} \subseteq \RR^n$ generated by the
  $\ve s_i - \ve s_1$ has dimension~$m-1$,
\item the affine space
  $\langle\ve s_1 \upto \ve s_m\rangle_\mathrm{aff} \subseteq \RR^n$
  generated by the~$\ve s_i$ has dimension~$m - 1$,
\item the~$\ve s_i$ are in general linear position and $m - 1 \le n$.
\end{itemize}
We will see in \cref{sDuality} that making this assumption or the
(similar) one made during \cref{sRank} provides sufficient
generality. With $k := n - (m - 1)$ (which may be any integer
between~$0$ and~$n$), we choose an orthonormal basis
$\ve w_1 \upto \ve w_k$ of the orthogonal complement
$U^\perp \subseteq \RR^n$. So with
$\gamma_j := \langle\ve s_i,\ve w_j\rangle$ (the Euclidean inner
product), which is
independent of~$i$, the vectors $\left(\begin{smallmatrix} \ve w_j \\
    2 \gamma_j\end{smallmatrix}\right)$ form a basis of the kernel of
$B$. Since $B$ has rank~$m$ there exist $\ve v \in \RR^n$ and
$\beta \in \RR$ such that
\begin{equation} \label{eqVbeta}
  B \cdot
  \begin{pmatrix}
    \ve v \\ \beta
  \end{pmatrix} =
  \begin{pmatrix}
    \lVert\ve s_1\rVert^2 - t_1^2 \\
    \vdots \\
    \lVert\ve s_m\rVert^2 - t_m^2
  \end{pmatrix}.  
\end{equation}
Moreover, we can and will assume
$\langle\ve v,\ve w_j\rangle = \gamma_j$ for all~$j$, which
makes~$\ve v$ and~$\beta$ unique. Notice that
$\langle\ve v,\ve w_j\rangle = \gamma_j$ means that
\begin{equation} \label{eqVsi}%
  \ve v \in \langle\ve s_1 \upto \ve s_m\rangle_\mathrm{aff}.
\end{equation}

In our case the matrix $A$ from~\cref{eqA} takes the form
\begin{equation} \label{eqAGPS}%
  A =
  \begin{pmatrix}
    -2 t_1 & 2 \ve s_1^T & -1 \\
    \vdots & \vdots & \vdots \\
    - 2 t_m & 2 \ve s_m^T & -1
  \end{pmatrix} \in \RR^{m \times (n+2)}.    
\end{equation}
Having $B$ as a submatrix, $A$ also has rank~$m$, so there exist
$\ve u \in \RR^n$ and $\alpha \in \RR$ such that
\begin{equation} \label{eqUalpha}%
  A \cdot
  \begin{pmatrix}
    1 \\ \ve u \\ 2 \alpha
  \end{pmatrix} = \ve 0 \quad \text{and} \quad \langle\ve u,\ve
  w_j\rangle = 0 \ \text{for all} \ j.
\end{equation}
Again, these conditions make~$\ve u$ and~$\alpha$ unique, and
$\ve u \in U$.

\begin{theorem}[The set of solutions
  of~\cref{eqGPSsquared}] \label{tSol}%
  Let $\ve s_1 \upto \ve s_m \in \RR^n$ be in general linear position,
  $m \le n+1$, and let $t_1 \upto t_m \in \RR$. Form
  $\ve u,\ve v,\ve w_1 \upto \ve w_k \in \RR^n$ and
  $\alpha,\beta,\gamma_1 \upto \gamma_k \in \RR$ as above. Then the
  set of solutions of~\cref{eqGPSsquared} is
  \[
    \mathfrak X = \Bigl\{(b,\ve x) \in \RR \times \RR^n \ \Bigl| \ \ve
    x = \ve v + b \ve u + \sum_{j=1}^k y_j \ve w_j \ \text{and} \ \
    g(b,y_1 \upto y_k) = 0\Bigr\},
  \]
  where
  \[
    g(b,y_1 \upto y_k) := \sum_{j=1}^k y_j^2 + \bigl(\lVert\ve
    u\rVert^2 - 1\bigr) b^2 + 2 \bigl(\langle\ve u,\ve v\rangle -
    \alpha\bigr) b + \lVert\ve v\rVert^2 -\beta.
  \]
\end{theorem}

\begin{proof}
  It follows directly from our definitions that
  $\left(\begin{smallmatrix} 1 \\ \ve u \\
      2 \alpha\end{smallmatrix}\right)$ and the
  $\left(\begin{smallmatrix} 0 \\ \ve w_j \\
      2 \gamma_j\end{smallmatrix}\right)$ ($j = 1 \upto k$) form a
  basis of the kernel of $A$, and that the solution space $L$ of the
  linear system~\cref{eqLinear} is
  \[
    L = \Bigl\{
    \begin{pmatrix}
      0 \\ \ve v \\ \beta
    \end{pmatrix} + b
    \begin{pmatrix}
      1 \\ \ve u \\ 2 \alpha
    \end{pmatrix} + \sum_{j=1}^k y_j
    \begin{pmatrix}
      0 \\ \ve w_j \\ 2 \gamma_j
    \end{pmatrix} \ \Bigl| \ b,y_1 \upto y_k \in \RR\Bigr\}.
  \]
  So we need to apply the equation~\cref{eqQuadratic} to
  $\vh x = \left(\begin{smallmatrix} b \\ \ve v + b \ve u + \sum y_j
      \ve w_j\end{smallmatrix}\right)$ and $\lambda = \beta + 2 \alpha
  b + 2 \sum y_j \gamma_j$, and this yields the condition $g(b,y_1
  \upto y_k) = 0$.
\end{proof}

Let us write
\begin{equation} \label{eqQsol}%
  \Sol := \bigl\{\ve x \in \RR^n \mid (b,\ve x) \in \mathfrak X \
  \text{for some} \ b \in \RR\bigr\}  
\end{equation}
for the set of points~$\ve x$ that come from a solution $(b,\ve x)$ of
\cref{eqGPSsquared}. It is clear from \cref{tSol} that $\Sol$ is a
quadric, which may be degenerate in several ways, within the affine
subspace
$\ve v + \langle\ve u,\ve w_1 \upto \ve w_k\rangle_\mathrm{lin}$. We
call $\Sol$ the \df{quadric of solutions}, and we will classify the
types of quadric that occur in \cref{sShape}. The following
\cref{exSphere1,pCollinear} look at two degenerate cases.

\begin{ex} \label{exSphere1}%
  Can it happen that $\ve u = \ve 0$? By~\cref{eqUalpha}, this implies
  $t_i = - \alpha$, so the~$t_i$ have to be all equal. Conversely, if
  all~$t_i$ are equal, then $\ve u := \ve 0$ and $\alpha := -t_i$
  satisfy~\cref{eqUalpha}. So $\ve u = \ve 0$ precisely if all~$t_i$
  are equal. This includes the case $m = 1$. From now on we assume
  that $\ve u = \ve 0$.
  Then~\cref{eqVbeta} can be rewritten as
  $2 \langle\ve s_i,\ve v\rangle - \lVert\ve s_i\rVert^2 = \beta -
  \alpha^2$ for all~$i$, so
  \begin{equation} \label{eqSphere}%
    \lVert\ve s_i - \ve v\rVert^2 = \lVert\ve v\rVert^2 - \beta +
    \alpha^2 =: r^2.    
  \end{equation}
  So also all $\lVert\ve s_i - \ve v\rVert$ are equal, which means
  that the~$\ve s_i$ lie on a sphere around~$\ve v$ with
  radius~$r$. Can it happen that $r = 0$? Not if $m > 1$, since
  $r = 0$ would imply $\ve s_1 = \ve s_2$, contradicting the general
  linear position hypothesis. But if $m = 1$, then~\cref{eqVsi}
  implies $\ve v = \ve s_1$, so then $r = 0$. Let us now apply
  \cref{tSol}.  In our situation $g(b,y_1 \upto y_k)$ specializes to
  \[
    g(b,y_1 \upto y_k) = \sum_{j=1}^k y_j^2 - b^2 - 2 \alpha b +
    \lVert\ve v\rVert^2 -\beta = \sum_{j=1}^k y_j^2 - (b + \alpha)^2 +
    r^2,
  \]
  so
  \begin{equation} \label{eqSolU0}%
    \mathfrak X = \Bigl\{(b,\ve x) \in \RR \times \RR^n \ \Bigl| \ \ve
    x = \ve v + \sum_{j=1}^k y_j \ve w_j, y_j \in \RR \
    \text{arbitrary, and} \ b = -\alpha \pm \sqrt{\textstyle r^2 +
      \sum_{j=1}^k y_j^2}\Bigr\}.
  \end{equation}
  Thus
  \begin{equation} \label{eqSolSphere}%
    \Sol = \ve v + U^\perp \quad \text{with} \quad U = \langle\ve s_i
    - \ve s_1 \mid 2 \le i \le m\rangle_\mathrm{lin}.
  \end{equation}
  So here the quadric degenerates to a $k$-dimensional affine
  subspace. The situation is illustrated in \cref{fSolSat2}. If
  $k = 0$, then $\Sol = \{\ve v\}$ and
  $\mathfrak X = \{(-\alpha + r,\ve v),(-\alpha - r,\ve v)\}$, and if
  $m = 1$, then $\Sol = \RR^n$.

  Taking a glance at~\cref{eqIneq} and comparing with~\cref{eqSolU0},
  we see that the solution(s) $(b,\ve x) \in \mathfrak X$ with $b = -\alpha - \sqrt{r^2 + \sum_{j=1}^k y_j^2}$ satisfy the original global positioning equations~\cref{eqGPS}, while the one where the square root is added does not.
\end{ex}

\begin{lemma}[Two satellites in view] \label{lM2}%
  In the situation of \cref{tSol}, assume $m = 2$. Then with $d :=
  \lVert\ve s_1 - \ve s_2\rVert$, the
  vector~$\ve u$ and two solutions of~\cref{eqGPSsquared} (but not
  necessarily the only ones) are
  \[
    \ve u = \frac{t_1 - t_2}{d^2} (\ve s_1 - \ve s_2), \ \text{so} \
    \lVert\ve u\rVert = \frac{|t_1 - t_2|}{d}, \ \ve x = \frac{\ve s_1
      + \ve s_2 \pm d \cdot \ve u}{2}, \ b = \frac{t_1 + t_2 \pm
      d}{2}.
  \]
\end{lemma}

\begin{proof}
  We have
  $\ve u \in U = \langle\ve s_1 - \ve s_2\rangle_\mathrm{lin}$, hence
  $\ve u = \delta (\ve s_1 - \ve s_2)$ with $\delta \in
  \RR$. Moreover, \cref{eqUalpha} tells us
  $-t_i + \langle\ve s_i,\ve u\rangle - \alpha = 0$, so
  $\langle\ve s_1 - \ve s_2,\ve u\rangle = t_1 - t_2$ and
  $\delta = \frac{t_1 - t_2}{d^2}$. This implies the formula
  for~$\ve u$. The solutions $(b,\ve x)$ can be verified by
  substituting them into~\cref{eqGPSsquared}.
\end{proof}

The following result deals with the situation that the coefficients
of~$b^2$, $b$ and the constant coefficient in~$g(b,y_1 \upto y_k)$ are
all zero. As it turns out, this is the most idiosyncratic case
occurring in this paper.

\begin{prop}[A degenerate case] \label{pCollinear}%
  In the situation of \cref{tSol}, we have $g(b,0 \upto 0) = 0$ for
  all~$b$ if and only if $m = 2$ and
  $|t_1 - t_2| = \lVert\ve s_1 - \ve s_2\rVert$. In this case
  \[
    \mathfrak X = \bigl\{(b,\ve x) \mid b \in \RR \ \text{arbitrary},
    \ \ve x = \ve v + b \ve u\bigr\},
  \]
  so $\Sol$ is a line. Moreover, the $(t,\ve s) \in \RR \times \RR^n$
  such that $\lVert\ve s - \ve x\lVert^2 - (t - b)^2 = 0$ for all
  $(b,\ve x) \in \mathfrak X$ are exactly those with
  $\ve s = \ve v + t \ve u$. So all the~$\ve s$ occuring in such pairs
  $(t,\ve s)$ are collinear. In particular, $\ve s_i = \ve v + t_i \ve
  u$ for $i = 1,2$.
\end{prop}

\begin{proof}
  We start by assuming $m = 2$ and
  $|t_1 - t_2| = \lVert\ve s_1 - \ve s_2\rVert$. Take
  $\ve x \in \langle\ve s_1,\ve s_2\rangle_\mathrm{aff}$, so
  $\ve x = \lambda_1 \ve s_1 + \lambda_2 \ve s_2$ with
  $\lambda_1 + \lambda_2 = 1$. With
  $b := \lambda_1 t_1 + \lambda_2 t_2$ we have
  \[
    \lVert\ve x - \ve s_i\rVert^2 - (b - t_i)^2 = (\lambda_i - 1)^2
    \lVert\ve s_1 - \ve s_2\rVert^2 - (\lambda_i - 1)^2 (t_1 - t_2)^2
    = 0,
  \]
  so $(b,\ve x) \in \mathfrak X$. Therefore
  $\langle\ve s_1,\ve s_2\rangle_\mathrm{aff} \subseteq \Sol$. Because
  $m = 2$ we have
  $\langle\ve s_1,\ve s_2\rangle_\mathrm{aff} = \{\ve v + b \cdot \ve
  u \mid b \in \RR\}$, so \cref{tSol} implies that
  $g(b,0 \upto 0) = 0$ for all~$b$.
  
  Conversely, assume that $g(b,0 \upto 0) = 0$ for all~$b$. Then by
  \cref{tSol} the solution set $\mathfrak X$ is as asserted in the
  proposition. The assumption also implies $\lVert\ve u\rVert =
  1$. Let $(t,\ve s) \in \RR \times \RR^n$ such that
  $\lVert\ve s - \ve x\lVert^2 - (t - b)^2 = 0$ for all
  $(b,\ve x) \in \mathfrak X$. Then for all $b \in \RR$ we have
  \[
    0 = \lVert\ve s - \ve v - b \ve u\lVert^2 - (t - b)^2 = \lVert\ve
    s - \ve v\rVert^2 - 2\langle\ve s - \ve v,\ve u\rangle b - t^2 + 2
    t b,
  \]
  so $t = \langle\ve u,\ve s - \ve v\rangle$ and
  $\lVert\ve s - \ve v\lVert^2 = \langle\ve u,\ve s - \ve
  v\rangle^2$. The second equation together with
  $\lVert\ve u\rVert = 1$ implies
  $\ve s - \ve v = \langle\ve u,\ve s - \ve v\rangle \ve u = t \ve u$,
  so $\ve s = \ve v + t \ve u$.
  Conversely, $\ve s = \ve v + t \ve u$ implies
  $\lVert\ve s - \ve x\lVert^2 - (t - b)^2 = 0$ for all
  $(b,\ve x) \in \mathfrak X$. Since the condition
  $\lVert\ve s - \ve x\lVert^2 - (t - b)^2 = 0$ for all
  $(b,\ve x) \in \mathfrak X$ applies in particular to
  $(t,\ve s) = (t_i,\ve s_i)$, we obtain $\ve s_i = \ve v + t_i \ve u$
  for $i = 1 \upto m$, so $m \le 2$ by the general linear position
  hypothesis. Moreover, $\lVert\ve u\rVert \ne 0$ implies $m > 1$ by
  \cref{exSphere1}, so $m = 2$. By \cref{lM2} we get
  $\lVert\ve s_1 - \ve s_2\rVert = |t_1 - t_2|$.
\end{proof}

\section{The quadric of satellites} \label{sQsat}

Having gotten a formula for the solution set $\mathfrak X$
of~\cref{eqGPSsquared} in \cref{tSol}, in this section we reverse our
point of view by considering the set
\begin{equation} \label{eqS}%
  \mathcal S := \Bigl\{(t,\ve s) \in \RR \times \RR^n \ \Bigl| \
  \lVert\ve s - \ve x\rVert^2 = (t - b)^2 \ \text{for all} \ (b,\ve x)
  \in \mathfrak X\Bigr\}.
\end{equation}
So clearly $(t_i,\ve s_i) \in \mathcal S$ for all~$i$. We continue to
use the definitions and notation from the previous section. The
following result determines $\mathcal S$ under the additional
hypothesis that~\cref{eqGPSsquared} has at least two solutions.

\begin{theorem}[The locus of satellites] \label{tSat}%
  Let $\ve s_1 \upto \ve s_m \in \RR^n$ be in general linear position,
  $m \le n+1$, and let $t_1 \upto t_m \in \RR$. Form
  $\ve u,\ve v,\ve w_1 \upto \ve w_k \in \RR^n$ and
  $\alpha,\beta,\gamma_1 \upto \gamma_k \in \RR$ as defined before
  \cref{tSol}. Set
  \[
    \mathcal S' := \Bigl\{(t,\ve s) \in \RR \times \RR^n \ \Bigl| \
    \langle\ve s,\ve w_j\rangle = \gamma_j \ (j = 1 \upto k), \ t =
    \langle\ve u,\ve s\rangle - \alpha, \ \text{and} \ \ \tilde h(\ve
    s) = 0\Bigr\},
  \]
  where
  \[
    \tilde h(\ve s) := \lVert\ve s\rVert^2 - \langle\ve u,\ve
    s\rangle^2 + 2 \langle\alpha \ve u - \ve v,\ve s\rangle + \beta -
    \alpha^2.
  \]
  Notice that the condition
  ``$\langle\ve s,\ve w_j\rangle = \gamma_j$'' in the definition of
  $\mathcal S'$ means that
  $\ve s \in \langle\ve s_1 \upto \ve s_m\rangle_\mathrm{aff}$.
  \begin{enumerate}[label=(\alph*)]
  \item \label{tSatA} For $i = 1 \upto m$ we have
    $(t_i,\ve s_i) \in \mathcal S'$.
  \item \label{tSatB} $\mathcal S' \subseteq \mathcal S$.
  \item \label{tSatC} If $|\mathfrak X| > 1$, then
    $\mathcal S = \mathcal S'$.
  \end{enumerate}
\end{theorem}

\begin{rem} \label{rX1}%
  The cases $|\mathfrak X| = 1$ and $|\mathfrak X| = 0$ are easily
  dealt with: If $\mathfrak X = \{(b_0,\ve x_0)\}$, then
  \[
    \mathcal S = \bigl\{(t,\ve s) \mid \ve s \in \RR^n \ \text{and} \
    \ t = b_0 \pm \lVert\ve s - \ve x_0\rVert\bigr\}. 
  \]
  If $\mathfrak X = \emptyset$ then $\mathcal S = \RR \times \RR^n$.
\end{rem}

The following lemma is the heart of the proof of \cref{tSat}.

\begin{lemma} \label{lSat}%
  In the situation of \cref{tSat} let $(t,\ve s) \in \RR \times \RR^n$
  and form the matrix $A'\in \RR^{(m+1) \times (n+2)}$ by appending
  the row $(-2 t,2 \ve s^T,-1)$ at the bottom of $A$ as defined
  by~\cref{eqAGPS}.  Let $\mathfrak X'$ be the set of solutions
  $(b,\ve x)$ of the system~\cref{eqGPSsquared}, enlarged by adding
  the pair $(t_{m+1},\ve s_{m+1}) := (t,\ve s)$.
  \begin{enumerate}[label=(\alph*)]
  \item \label{lSatA} If $\rank(A') = m$ then
    $\langle\ve s,\ve w_j\rangle = \gamma_j$ for $j = 1 \upto k$ and
    $t = \langle\ve u,\ve s\rangle - \alpha$.
  \item \label{lSatB} If
    $\langle\ve s,\ve w_j\rangle = \gamma_j$ for $j = 1 \upto k$ and
    $t = \langle\ve u,\ve s\rangle - \alpha$, then
    \[
      \mathfrak X' = \left\{
        \begin{array}{lll}
          \mathfrak X & \text{if} & \tilde h(\ve s) = 0 \\
        \emptyset & \text{if} & \tilde h(\ve s) \ne 0
        \end{array}\right..
    \]
  \end{enumerate}
\end{lemma}

\begin{proof}
  \begin{enumerate}
  \item[\ref{lSatA}] Since $\rank(A) = \rank(A')$, both matrices share
    the same kernel. In particular, the vectors
    $\left(\begin{smallmatrix} 0 \\ \ve w_j \\ 2
        \gamma_j\end{smallmatrix}\right)$ (for $k = 1 \upto k$) and
    $\left(\begin{smallmatrix} 1 \\ \ve u \\
        2 \alpha\end{smallmatrix}\right)$ lie in the kernel of
    $A'$. This implies $\langle\ve s,\ve w_j\rangle - \gamma_j = 0$
    and $-t + \langle\ve s,\ve u\rangle - \alpha = 0$.
  \item[\ref{lSatB}] Let $b,y_1 \upto y_k \in \RR$ and set
    $\ve x := \ve v + b \ve u + \sum_{j=1}^k y_j \ve w_j$. Then
    \begin{multline} \label{eqGh}%
      \lVert\ve s - \ve x\rVert^2 - (t - b)^2 =
      \bigl\langle\textstyle\ve v + b \ve u + \sum_{j=1}^k y_j \ve w_j
      - \ve s,\ve v + b \ve u + \sum_{j=1}^k y_j \ve w_j - \ve
      s\bigr\rangle \\
      - \bigl(\langle\ve u,\ve s\rangle - \alpha - b\bigr)^2 =
      \lVert\ve v\rVert^2 + 2 \langle\ve u,\ve v\rangle b + 2
      \sum_{j=1}^k y_j \gamma_j - 2 \langle\ve v,\ve s\rangle +
      \lVert\ve u\rVert^2 b^2 -
      2 \langle\ve u,\ve s\rangle b \\
      + \sum_{j=1}^k y_j^2 - 2 \sum_{j=1}^k y_k \gamma_j + \lVert\ve
      s\rVert^2 - \langle\ve u,\ve s\rangle^2 + 2 \langle\ve u,\ve
      s\rangle \alpha
      + 2 \langle\ve u,\ve s\rangle b - \alpha^2 - 2 \alpha b - b^2 = \\
      \bigl(\lVert\ve u\rVert^2 - 1\bigr) b^2 + 2 \bigl(\langle\ve
      u,\ve v\rangle - \alpha\bigr) b + \lVert\ve v\rVert^2 - 2
      \langle\ve v,\ve s\rangle + \sum_{j=1}^k y_j^2 + \lVert\ve
      s\rVert^2 - \langle\ve u,\ve s\rangle^2 + 2 \langle\alpha \ve
      u,\ve s\rangle -
      \alpha^2 \\
      = g(b,y_1 \upto y_k) + \tilde h(\ve s).
    \end{multline}
    Suppose $\tilde h(\ve s) = 0$. If $(b,\ve x) \in \mathfrak X$ then
    by \cref{tSol}, $\ve x$ can be written as above and
    $g(b,y_1 \upto y_k) = 0$. So~\cref{eqGh} tells us that
    $\lVert\textstyle\ve s - \ve x\rVert^2 - (t - b)^2 = 0$, so
    $(b,\ve x) \in \mathfrak X'$. Since
    $\mathfrak X' \subseteq \mathfrak X$ holds in any case, we have
    equality. Now suppose $\mathfrak X' \ne \emptyset$, so there is
    $(b,\ve x) \in \mathfrak X'$. Thus $(b,\ve x) \in \mathfrak X$,
    and by \cref{tSol},
    $\ve x = \ve v + b \ve u + \sum_{j=1}^k y_j \ve w_j$ with
    $g(b,y_1 \upto y_m) = 0$. We have
    $\lVert\ve s - \ve x\rVert^2 = (t - b)^2$, so \cref{eqGh} shows
    $\tilde h(\ve s) = 0$. Therefore $\mathfrak X' = \emptyset$ if
    $\tilde h(\ve s) \ne 0$. \endproof
  \end{enumerate}
\end{proof}

\begin{proof}[Proof of \cref{tSat}]
  \begin{enumerate}
      \item[\ref{tSatA}] Let $i \in \{1 \upto m\}$. Then 
      $\langle\ve s_i,\ve w_j\rangle = \gamma_j$ by the definition of
      the~$\gamma_j$, and $t_i = \langle\ve u,\ve s_i\rangle - \alpha$ 
      follows from~\cref{eqUalpha}. From~\cref{eqVbeta} we get 
      $\langle\ve v,\ve s_i\rangle = \frac{1}{2}\bigl(\lVert\ve 
      s_i\rVert^2 - t_i + \beta\bigr)$. Therefore
      \[
      \tilde h(\ve s_i) = \lVert\ve s_i\rVert^2 - (t_i + \alpha)^2
      + 2 \alpha (t_i + \alpha) - \bigl(\lVert\ve s_i\rVert^2 - t_i^2 +
      \beta\bigr) + \beta - \alpha^2 = 0.
      \]
      So we have shown $(t_i,\ve s_i) \in \mathcal S'$.
    \item[\ref{tSatB}] Let $(t,\ve s) \in \mathcal S'$. Then
      \cref{lSat}\ref{lSatB} tells us that
      $\mathfrak X' = \mathfrak X$, so
      $\lVert\ve s - \ve x\rVert^2 = (t - b)^2$ for every
      $(b,\ve x) \in \mathfrak X$. Therefore
      $(t,\ve s) \in \mathcal S$, and the inclusion
      $\mathcal S' \subseteq \mathcal S$ is established.
    \item[\ref{tSatC}] With $g(b,y_1 \upto y_k)$ as defined in
      \cref{tSol}, it is convenient to write
      $g_0(b) := g(b,0 \upto 0)$, which is a polynomial of degree
      $\le 2$. Let $b \in \RR$. We see from \cref{tSol} that in the
      case $k > 0$, there exists an $\ve x$ with
      $(b,\ve x) \in \mathfrak X$ if and only if $g_0(b) \le 0$. In
      the case $k = 0$ the condition is $g_0(b) = 0$. In both cases,
      if $g_0(b) = 0$, then~$\ve x$ is unique. Thus in the case
      $k = 0$ the hypothesis $|\mathfrak X| > 1$ implies that there
      are real numbers $b_1 \ne b_2$ such that $g_0(b_i) = 0$ for
      both~$i$. In addition, \cref{tSol} shows that
      $\langle\mathfrak X\rangle_\mathrm{aff} \subseteq \RR^{n+1}$,
      the affine subspace generated by $\mathfrak X$, has
      dimension~$1$. On the other hand, in the case $k > 0$, we see
      that either $g_0(b) = 0$ for all~$b$ or there are real numbers
      $b_1 \ne b_2$ such that $g_0(b_i) < 0$ for both~$i$. In the
      first subcase we can pick $b_1 \ne b_2$ with $g_0(b_i) = 0$, and
      in the second we obtain
      $\dim\bigl(\langle\mathfrak X\rangle_\mathrm{aff}\bigr) =
      k+1$. What we need to remember from this discussion is the
      following: First, there exist real numbers $b_1 \ne b_2$ and
      points $\ve x_i \in \RR^n$ such that
      $(b_i,\ve x_i) \in \mathfrak X$ for both~$i$. And second,
      $\dim\bigl(\langle\mathfrak X\rangle_\mathrm{aff}\bigr) = k+1$
      or else $g_0$ is identically zero.

      Now let $(t,\ve s) \in \mathcal S$. Form matrices
      $A' \in \RR^{(m+1) \times (n+2)}$ and
      $B' \in \RR^{(m+1) \times (n+1)}$ by appending the rows
      $(-2 t,2 \ve s^T,-1)$ and $(2 \ve s^T,-1)$ at the bottom of $A$
      and $B$, respectively. $A'$ and $B'$ clearly have rank $\ge
      m$. We claim that $\rank(A') = \rank(B') = m$. We first treat
      $B'$, and first consider the case that $g_0(b) = 0$ for
      all~$b$. This case was treated in \cref{pCollinear}, which tells
      us that $m = 2$ and that $\ve s_1,\ve s_2$, and~$\ve s$ are
      collinear. So we have $\rank(B') = 2 = m$. In the case that
      $g_0$ is not identically zero we have
      $\dim\bigl(\langle\mathfrak X\rangle_\mathrm{aff}\bigr) =
      k+1$. By way of contradiction, let us assume
      $\rank(B') = m + 1$. Then we can apply \cref{tSol} to the
      enlarged system~\cref{eqGPSsquared} with
      $(t_{m+1},\ve s_{m+1}) := (t,\ve s)$, and this results in a set
      of solutions $\mathfrak X'$ with
      $\dim\bigl(\langle\mathfrak X'\rangle_\mathrm{aff}\bigr) < k +
      1$. But $\mathfrak X = \mathfrak X'$ since
      $(\ve s,t) \in \mathcal S$, so this is a contradiction. Thus
      also in this case $\rank(B') = m$.

      To deal with the matrix $A'$ we use the
      $(b_i,\ve x_i) \in \mathfrak X$, $i = 1,2$. Since
      $(\ve s,t) \in \mathcal S$, the $(b_i,\ve x_i)$ are also
      solutions of the enlarged system. Therefore they also satisfy
      the matrix version~\cref{eqMatrix}, and we get
      \[
        A' \cdot
        \begin{pmatrix}
          b_1 - b_2 \\ \ve x_1 - \ve x_2 \\ \lVert\ve x_1\rVert^2 -
          b_1^2 - \lVert\ve x_2\rVert^2 + b_2^2
        \end{pmatrix} =
        \begin{pmatrix}
          \lVert\ve s_1\rVert^2 - t_1^2 \\
          \vdots \\
          \lVert\ve s_m\rVert^2 - t_m^2 \\
          \lVert\ve s\rVert^2 - t^2
        \end{pmatrix} -
        \begin{pmatrix}
          \lVert\ve s_1\rVert^2 - t_1^2 \\
          \vdots \\
          \lVert\ve s_m\rVert^2 - t_m^2 \\
          \lVert\ve s\rVert^2 - t^2
        \end{pmatrix} = \ve 0.
      \]
      Since $b_1 \ne b_2$, this shows that the first column of $A'$
      lies in the span of the other columns. But $A'$ is obtained from
      $B'$ by adding the first column, so $\rank(A') = \rank(B') = m$.

      Now \cref{lSat} tells us that
      $\langle\ve s,\ve w_j\rangle = \gamma_j$ for $j = 1 \upto k$ and
      $t = \langle\ve u,\ve s\rangle - \alpha$, and, since
      $\emptyset \ne \mathfrak X = \mathfrak X'$, that
      $\tilde h(\ve s) = 0$.  So $(t,\ve s) \in \mathcal S'$, which
      completes the proof. \endproof
  \end{enumerate}
\end{proof}

Analogously to the quadric $\Sol$ of solutions we also consider the
\df{quadric of satellites}, defined as
\begin{equation} \label{eqQsat}%
  \Sat := \bigl\{\ve s \in \RR^n \mid (t,\ve s) \in \mathcal S \
  \text{for some} \ t \in \RR\bigr\},  
\end{equation}
with $\mathcal S$ from~\cref{eqS}. It should be emphasized that even though the satellite positions~$\ve s_i$ lie on $\Sat$, this does not mean that $\Sat$ is determined by the~$\ve s_i$. Rather, it is determined by the~$\ve s_i$ {\em together} with the~$t_i$. For instance, \cref{exQuadrics} shows that for the same~$\ve s_i$, very different sorts of quadrics occur.  If $|\mathfrak X| > 1$, then by
\cref{tSat} this consists of all points $\ve s$ in the affine span of
$\ve s_1 \upto \ve s_m$ that satisfy $\tilde h(\ve s) = 0$. So $\Sat$
is a quadric within
$\langle\ve s_1 \upto \ve s_m\rangle_\mathrm{aff}$, which may be
degenerate. For example, in the case of \cref{pCollinear}, $\Sat$ and
$\Sol$ are both equal to the line $\RR \cdot\ve u + \ve v$.

We give two examples.

\begin{ex} \label{exSphere2}%
  As a follow-up to \cref{exSphere1} we look at the case
  $\ve u = \ve 0$. We have already seen that this is equivalent to
  all~$t_i$ being equal, and it implies that the~$\ve s_i$ lie on a
  sphere with radius
  $r = \sqrt{\lVert\ve v\rVert^2 - \beta + \alpha^2}$ around the
  point~$\ve v$. Applying \cref{tSat}, we get
  \[
    \tilde h(\ve s) = \lVert\ve s\rVert^2 - 2\langle\ve v,\ve s\rangle
    + \beta - \alpha^2 = \lVert\ve s - \ve v\rVert^2 - r^2,
  \]
  so $\Sat$ is the sphere inside $\langle\ve s_1 \upto \ve
  s_m\rangle_\mathrm{aff}$ with radius~$r$ and center~$\ve
  v$. Moreover,
  \[
    \mathcal S = \bigl\{(-\alpha,\ve s) \mid \ve s \in \langle\ve s_1
    \upto \ve s_m\rangle_\mathrm{aff} \quad \text{and} \quad \lVert\ve
    s - \ve v\rVert = r\bigl\}.
  \]
  Recall that~$-\alpha$ is the common value of all~$t_i$. See
  \cref{fSolSat2} for a picture.
\end{ex}

The next example shows that the hypothesis $|\mathfrak X| > 1$ in
\cref{tSat}\ref{tSatC} cannot be omitted, or replaced by 
$\mathfrak X \ne \emptyset$.

\begin{ex} \label{exCone}%
  Take the points
  \[
    \ve s_1 = (1,0,0)^T, \ \ve s_2 = (2,0,0)^T, \ \text{and} \ \ve s_3 =
    (0,1,0)^T \in \RR^3
  \]
  together with $t_1 = 1$, $t_2 = 2$, and $t_3 = 1$. So
  \[
    B =
    \begin{pmatrix}
      2 & 0 & 0 & -1 \\
      4 & 0 & 0 & -1 \\
      0 & 2 & 0 & -1
    \end{pmatrix} \quad \text{and} \quad
    A =
    \begin{pmatrix}
      -2 & 2 & 0 & 0 & -1 \\
      -4 & 4 & 0 & 0 & -1 \\
      -2 & 0 & 2 & 0 & -1
    \end{pmatrix}.
  \]
  We get $k = 1$, $\ve w_1 = (0,0,1)^T$, $\gamma_1 = 0$, and, since
  $t_i^2 = \lVert\ve s_i\rVert^2$, $\ve v = \ve 0$ and $\beta =
  0$. Moreover, $\ve u = (1,1,0)^T$ and $\alpha = 0$. Applying
  \cref{tSol}, we obtain $g(b,y) = y^2 + b^2$, so
  $\mathfrak X = \{(0,\ve 0)\}$ consists of a single solution. So by
  \cref{rX1} we have
  $\mathcal S = \{(t,\ve s) \mid \ve s \in \RR^3, \ t = \pm \lVert\ve
  s\rVert\}$. However, evaluating $\tilde h(\ve s)$ from \cref{tSat}
  for a point $\ve s = (z_1,z_2,z_3)^T$ yields
  \[
    \tilde h(\ve s) = z_1^2 + z_2^2 + z_3^2 - (z_1 + z_2)^2 = z_3^2 -
    2 z_1 z_2.
  \]
  Since the set $\mathcal S'$ in \cref{tSat} has the additional
  condition $\langle\ve s,\ve w_1\rangle = \gamma_1$, we obtain
  \[
    \mathcal S' := \Bigl\{\bigl(z_1 + z_2,\left(\begin{smallmatrix}
        z_1 \\ z_2 \\ 0\end{smallmatrix}\right)\bigr) \ \bigl| \ z_1
    z_2 = 0\Bigr\}.
  \]
  So the $(t_i,\ve s_i)$ lie in $\mathcal S'$, $\mathcal S'$ is
  contained in $\mathcal S$ (as predicted by \cref{tSat}), but it is
  far from being equal to $\mathcal S$.
\end{ex}

\section{What do the quadrics look like?} \label{sShape}%

In this section we take a closer look at the quadrics $\Sol$ and
$\Sat$ defined in~\cref{eqQsol,eqQsat}. A quick glance at
\cref{fQuadrics} should give readers a good idea of the upshot of this
section. We keep the notation of the previous sections and make the
following three assumptions:
\begin{enumerate}[label=(\arabic*)]
\item \label{A1} the points $\ve s_1 \upto \ve s_m \in \RR^n$ are in
  general linear position, with $m \le n+1$,
\item \label{A2} the system~\cref{eqGPSsquared} has at least two
  solutions, and
\item \label{A3} the~$t_i$ are not all equal.
\end{enumerate}

By \cref{tSol}, $\Sol$ consists of all
$\ve x = \ve v + y_0 \ve u + \sum_{j=1}^k y_j \ve w_j$ such that
$g(y_0,y_1 \upto y_k) = 0$. (It is convenient to write~$y_0$ here
instead of~$b$ appearing in \cref{tSol}.) By \cref{exSphere1}, our
assumption~\ref{A3} implies $\ve u \ne \ve 0$. We first assume
$\Vert\ve u\rVert \ne 1$ and treat the case $\Vert\ve u\rVert = 1$
later. Set
\begin{equation} \label{eqEmd}%
  e := \lVert\ve u\rVert, \quad \mu := \frac{\langle\ve u,\ve v\rangle
    - \alpha}{e^2 - 1} \quad \text{and} \quad \rho := \frac{\bigl(
    \langle\ve u,\ve v\rangle - \alpha\bigr)^2}{e^2 - 1}
  -\lVert\ve v\rVert^2 + \beta.
\end{equation}
Then a short calculation shows
\begin{equation} \label{eqG}%
  g(y_0 \upto y_k) = \sum_{j=1}^k y_j^2 + \sgn(e^2 - 1)
  \Bigl(\sqrt{|e^2 - 1|} \bigl(y_0 + \mu\bigr)\Bigr)^2 - \rho.
\end{equation}
Since $|\mathfrak X| > 1$ this implies $\rho > 0$ if
$\sgn(e^2 - 1) = 1$.

Let us turn our attention to $\Sat$, which is given by
\cref{tSat}. The theorem says that
$\Sat \subseteq \langle\ve s_1 \upto \ve s_m\rangle_\mathrm{aff}$,
which is equal to $\ve v + U$ with
$U = \langle\ve s_i - \ve s_1 \mid 2 \le i \le
m\rangle_\mathrm{lin}$. We have $\ve u \in U$, so we can choose
$\ve w_1' \upto \ve w_{m-2}'$ such that the $\ve w_i'$ together with
$e^{-1} \ve u$ form an orthonormal basis of $U$. Then every
$\ve s \in \Sat$ can be written as
\begin{equation} \label{eqSq}%
  \ve s = \ve v + z_0 \ve u + \sum_{i=1}^{m-2} z_i \ve w_i'  
\end{equation}
with $z_0 \upto z_{m-2} \in \RR$. \cref{tSat} also has the condition
$\tilde h(\ve s) = 0$. Substituting the above~$\ve s$ into~$\tilde h$
and performing some calculation yields
\begin{multline} \label{eqH}%
  h(z_0 \upto z_{m-2}) := \tilde h\bigl(\ve v + z_0 \ve u + \textstyle
  \sum_{i=1}^{m-2} z_i \ve w_i'\bigr) = \\
  \sum_{i=1}^{m-2} z_i^2 - e^2 (e^2 - 1) z_0^2 - 2 e^2
  \bigl(\langle\ve u,\ve v\rangle - \alpha\bigr) z_0 -
  \bigl(\langle\ve u,\ve v\rangle - \alpha\bigr)^2 - \lVert\ve
  v\rVert^2 + \beta.
\end{multline}
As above, we first assume $e \ne 1$. Then, again after a bit
of calculation, we obtain
\begin{equation} \label{eqH2}%
  h(z_0 \upto z_{m-2}) = \sum_{i=1}^{m-2} z_i^2 - \sgn(e^2 - 1)
  \Bigl(e \sqrt{|e^2 -1|} (z_0 + \mu)\Bigr)^2 + \rho
\end{equation}
with~$\mu$ and~$\rho$ defined in~\cref{eqEmd}. By \cref{tSat} all
$(t_i,\ve s_i)$ lie in $\mathcal S$, so the~$\ve s_i$
satisfy~\cref{eqH2}. Therefore the assumtion~\ref{A3} implies
that~\cref{eqH2} has at least two solutions, so if
$\sgn(e^2 - 1) = -1$, then $\rho < 0$. Together with what we have
found after~\cref{eqG} we obtain
\[
  \sgn(e^2 - 1) = \sgn(\rho).
\]
We can now use~\cref{eqG,eqH2} to determine the types of quadric we
get and to pick out the parameters. The results can be found in
columns for the cases $e > 1$ and $0 < e < 1$ in \cref{taQuadrics}. A
few explanations and comments should be made:

\begin{enumerate}[label=(\arabic*)]
\item In the ``type'' row, ``spheroid'' is short for ``prolate
  spheroid,'' and ``hyperboloid'' is short for ``hyperboloid of
  revolution of two sheets.''
\item By ``semiaxis $a$'' we mean the distance from the vertices to
  the center, often called the major semiaxis. The minor semiaxis
  appears in the table as ``semiaxis $b$.''
\item The eccentricity of a prolate spheroid is
  $\sqrt{1 - \left(\frac{b}{a}\right)^2}$, and the eccentricity of a
  hyperboloid of revolution of two sheets is
  $\sqrt{1 + \left(\frac{b}{a}\right)^2}$. Both formulas are well
  known and can be found in~[\citenumber{Boutin:Kemper:2024}] for
  general dimension. The distance from the foci to the center, also
  known as the linear eccentricity, is easy to determine: it is the
  major semiaxis times the eccentricity.
\item If~$k = 0$, $\Sol$ consists of just two points, and likewise for
  $\Sat$ if $m = 2$. In these cases, calling the sets ``spheroid'' or
  ``hyperboloid'' and assigning foci and eccentricity to them may seem
  a bit fictitious. This could be remedied by regarding $\Sol$ and
  $\Sat$ not as mere sets of points but enriching them by additional
  data, such as the eccentricity.

  If $k = 1$ for $\Sol$ or $m = 3$ for $\Sat$, the more appropriate
  type labels would be ``ellipse'' and ``hyperbola,'' and the axis of
  symmetry is not uniquely determined by the mere set of points.
\end{enumerate}

\begin{table}[h]
  \addtolength{\arraycolsep}{10mm}%
  \renewcommand{\arraystretch}{1.6}
  \newcommand{\bx}[2][18mm]{\parbox[c]{#1}{\centering\smallskip #2\smallskip}}
  \begin{tabular}{|c|c|c|c|c|c|c|}
    \hline%
    \bf quadric & \multicolumn{3}{c|}{$\Sol$} &%
    \multicolumn{3}{c|}{$\Sat$} \\%
    \hline%
    \bf case & $e > 1$ & $0 < e < 1$ & \bx{$e = 1$, $3 \le m \le n$} &%
    $e > 1$ & $0 < e < 1$ & \bx{$e = 1$, $3 \le m \le n$} \\%
    \hline%
    \bf type & spheroid & hyperboloid & paraboloid & hyperboloid &%
    spheroid & paraboloid \\%
    \hline%
    \bx{\bf axis of symmetry}  & \multicolumn{6}{c|}%
    {$\ve v + \RR \ve u$} \\%
    \hline%
    \bf center & \multicolumn{2}{c|}{$\ve c := \ve v - \mu \ve u$} &%
    N/A & \multicolumn{2}{c|}{$\ve c := \ve v - \mu \ve u$} & N/A \\%
    \hline%
    \bf vertices & \multicolumn{2}{c|}%
      {\bx[30mm]{$\ve c \pm \sqrt{\frac{\rho}{e^2 - 1}}
      \cdot \ve u$}} & $\ve v + \lambda_1 \ve u$ & \multicolumn{2}{c|}%
      {\bx[30mm]{$\ve c \pm e^{-1} \sqrt{\frac{\rho}{e^2 - 1}}%
      \cdot \ve u$}} & $\ve v + \lambda_2 \ve u$ \\
    \hline%
    \bf semiaxis $a$ & \multicolumn{2}{c|}%
      {\bx[30mm]{$a := e \sqrt{\frac{\rho}{e^2 - 1}}$}} & N/A &%
    \multicolumn{2}{c|}{\bx[30mm]{$a := \sqrt{\frac{\rho}{e^2 - 1}}$}}%
    &N/A \\%
    \hline%
    \bf semiaxis $b$ & \multicolumn{2}{c|}{$b := \sqrt{|\rho|}$} & N/A &%
    \multicolumn{2}{c|}{$b := \sqrt{|\rho|}$} & N/A \\%
    \hline%
    \bf eccentricity & \multicolumn{3}{c|}{$e^{-1}$} &
    \multicolumn{3}{c|}{$e$} \\
    \hline%
    \bx{\bf semilatus rectum} & \multicolumn{2}{c|}{$e^{-1}
      \sqrt{\rho (e^2 - 1)}$} & $|\langle\ve u,\ve v\rangle -%
    \alpha|$ & \multicolumn{2}{c|}{$\sqrt{\rho (e^2 - 1)}$} &%
    $|\langle\ve u,\ve v\rangle -\alpha|$ \\%
    \hline%
    \bf foci & \multicolumn{2}{c|}%
      {\bx[30mm]{$\ve c \pm e^{-1} \sqrt{\frac{\rho}{e^2 - 1}}
      \cdot \ve u$}} & $\ve v + \lambda_2 \ve u$ & \multicolumn{2}{c|}%
      {\bx[30mm]{$\ve c \pm \sqrt{\frac{\rho}{e^2 - 1}}%
      \cdot \ve u$}} & $\ve v + \lambda_1 \ve u$ \\
    \hline%
  \end{tabular}
  \bigskip
  \caption{Type and parameters of the quadrics $\Sol$ and $\Sat$
    under the assumptions~\ref{A1}--\ref{A3}. The case of
    \cref{pCollinear} is excluded from the table.} \label{taQuadrics}
\end{table}

We now treat the case $e = 1$, in which the~$y_0^2$-part
and~$z_0^2$-part in $g(y_0 \upto y_k)$ and $h(z_0 \upto z_{m-2})$,
respectively, vanishes. A subcase is that also the linear part
vanishes. Then
\begin{equation} \label{eqCylinder}%
  g(y_0 \upto y_k) = \sum_{j=1}^k y_j^2 + \lVert\ve v\rVert^2 -
  \beta \quad \text{and} \quad h(z_0 \upto z_{m-2}) = \sum_{i=1}^{m-2}
  z_i^2 - \lVert\ve v\rVert^2 + \beta,  
\end{equation}
and our assumptions~\ref{A2} and~\ref{A3} on \cpageref{A2} imply
$\lVert\ve v\rVert^2 - \beta = 0$. So this is the case of
\cref{pCollinear}.

On the other hand, assuming that $e = 1$ but
$\langle\ve u,\ve v\rangle - \alpha \ne 0$, we set
\begin{equation} \label{eqLambda}%
  \lambda_1 := \frac{\beta - \lVert\ve v\rVert^2}{2\bigl(\langle\ve
    u,\ve v\rangle - \alpha\bigr)} \quad \text{and} \quad \lambda_2 :=
  \lambda_1 - \frac{\langle\ve u,\ve v\rangle - \alpha}{2}.  
\end{equation}
With this,
\begin{equation} \label{eqG2}%
  g(y_0 \upto y_k) = \sum_{j=1}^k y_j^2 + 2 \bigl(\langle\ve u,\ve
  v\rangle - \alpha\bigr) \bigl(y_0 - \lambda_1\bigr)
\end{equation}
and
\begin{equation} \label{eqH3}%
  h(z_0 \upto z_{m-2}) = \sum_{i=1}^{m-2} z_i^2 - 2 \bigl(\langle\ve
  u,\ve v\rangle - \alpha\bigr) \bigl(z_0 - \lambda_2\bigr).
\end{equation}
By our assumptions~\ref{A2} and~\ref{A3}, this implies that~$k$
and~$m - 2$ are positive, i.e., $3 \le m \le n$. If this is not
satisfied, then the case $e = 1$ and
$\langle\ve u,\ve v\rangle - \alpha \ne 0$ we are considering at the
moment cannot happen. From~\cref{eqG2,eqH3} we can now fill in the
remaining entries in \cref{taQuadrics}: Both $\Sol$ and $\Sat$ are
paraboloids of revolution (or parabolas). The positions of the
vertices are immediately clear. How to determinate the semilatus
rectum of a paraboloid is well known, as is the fact that the distance
between vertex and focus, also known as the focal length, is one-half
of the semilatus rectum (see, for example,
[\citenumber{Boutin:Kemper:2024}]).
To determine the focus, one has to move from the vertex along the axis
of symmetry in the correct direction, i.e., towards where the
paraboloid has its points. This yields the formulas given in
\cref{taQuadrics}.

Notice that the case of \cref{pCollinear} is not represented in
\cref{taQuadrics}, even though this case satisfies our
assumptions~\ref{A1}--\ref{A3}%
. In fact, in this case
$\Sol = \Sat = \ve v + \RR\ve u = \langle\ve s_1,\ve
s_2\rangle_\mathrm{aff}$, and except for the axis of symmetry, none of
the table entries is applicable.

The following theorem summarizes our findings.

\begin{theorem}[The quadric of solutions and the quadric of
  satellites] \label{tSolSat}%
  Let $\ve s_1 \upto \ve s_m \in \RR^n$ be in general linear position,
  $m \le n+1$, and let $t_1 \upto t_m \in \RR$ be numbers that are not
  all equal. Assume that~\cref{eqGPSsquared} has at least two
  solutions, and that $m > 2$ or
  $|t_1 - t_2| \ne \lVert\ve s_1 - \ve s_2\rVert$. With
  $k := n + 1 - m$ and with
  $U, \ve u,\ve v,\ve w_1 \upto \ve w_k,\ve w_1' \upto \ve w_{m-2}'$
  defined at the beginning of \cref{sQsol} and after~\cref{eqG}, we
  have:
  \begin{enumerate}[label=(\alph*)]
  \item \label{tSolSatA} The sets $\Sol$, defined in~\cref{eqQsol},
    and $\Sat$, defined in~\cref{eqQsat}, are given given by
    \[
      \Sol = \Bigl\{\ve v + y_0 \ve u + \textstyle \sum_{j=1}^k y_j \ve
      w_j \ \Bigl| \ y_0 \upto y_k \in \RR \ \text{such that} \
      g(y_0 \upto y_k) = 0\Bigr\}
    \]
    with $g(y_0 \upto y_k)$ given by~\cref{eqG} if
    $e := \lVert\ve u\rVert \ne 1$ and by~\cref{eqG2} if $e = 1$, and
    \[
      \Sat = \Bigl\{\ve v + z_0 \ve u + \textstyle \sum_{i=1}^{m-2} z_i \ve
      w_i' \ \Bigl| \ z_0 \upto z_{m-2} \in \RR \ \text{such that} \
      h(z_0 \upto z_{m-2}) = 0\Bigr\}
    \]
    with $h(z_0 \upto z_{m-2})$ given by~\cref{eqH2} if $e \ne 1$ and
    by~\cref{eqH3} if $e = 1$.
  \item \label{tSolSatB} The affine subspaces of $\RR^n$ spanned by
    $\Sol$ and $\Sat$ are
    \[
      \Asol := \langle\Sol\rangle_\mathrm{aff} = \ve v + \langle\ve
      u,\ve w_1 \upto \ve w_k\rangle_\mathrm{lin} = \ve v + \RR \ve u
      + U^\perp
    \]
    and
    \[
      \Asat := \langle\Sat\rangle_\mathrm{aff} = \ve v + \langle\ve
      u,\ve w_1' \upto \ve w_{m-2}'\rangle_\mathrm{lin} = \ve v + U =
      \langle\ve s_1 \upto \ve s_m\rangle_\mathrm{aff}.
    \]
    Their dimensions are~$k+1$ and~$m-1$. $\Asol$ and $\Asat$ are
    perpendicular, and their intersection is the line
    $\ve v + \RR \ve u$.
  \item \label{tSolSatC} Viewed as a subset of its affine span, each
    of $\Sol$ and $\Sat$ is a quadric with a focus. However, if
    $k = 0$ then $\Sol$ consists of only two points, and if $m = 2$
    the same is true for $\Sat$. Both quadrics share the same axis of
    symmetry, which is $\ve v + \RR \ve u = \Asol \cap \Asat$.
  \item \label{tSolSatD} The foci (or the single focus in the case of
    a paraboloid) of $\Sat$ are the vertices (or the vertex) of $\Sol$
    and vice versa. In particular, the foci of $\Sat$ are solutions
    of~\cref{eqGPSsquared}. The eccentricities of $\Sol$ and $\Sat$
    are reciprocal to each other.
  \item \label{tSolSatE} $\Sol$ meets $\Asat$ perpendicularly. More
    precisely, the tangent space of $\Sol$ at every intersection point
    with $\Asat$ (by \cref{tSolSatD}, these intersection points are
    the vertices of $\Sol$) is equal to the orthogonal complement
    $U^\perp$.
  \end{enumerate}
\end{theorem}

\begin{proof}
  All parts except for~\ref{tSolSatE} have been deduced before stating
  the theorem. For the proof of~\ref{tSolSatE}, let $T_{\ve x}$ be the
  tangent space at a point~$\ve x \in \Sol$. Then $T_{\ve x}$ is
  contained in the linear space associated to $\Asol$, which by
  \cref{tSolSatB} is $\RR \ve u + U^\perp$. If~$\ve x$ is a vertex of
  $\Sol$, it is geometrically clear, and can be verified by looking
  at~\cref{eqG} and~\cref{eqG2}, that a tangent vector is
  perpendicular to~$\ve u$ and thus lies in $U^\perp$. So for the
  vertices we have $T_{\ve x} \subseteq U^\perp$. Equality follows
  since
  $\dim(T_{\ve x}) = \dim(\Asol) - 1 \underset{\ref{tSolSatB}}{=}
  \dim(U^\perp)$.
\end{proof}

The theorem is illustrated in \cref{fQuadrics}, which shows the case
$n = m = 3$, so $k = 1$.

\begin{figure}[h]
  \begin{tikzpicture}
    \begin{axis}[%
      view={78}{20},%
      xlabel = x,
      ylabel = y,
      zlabel = z,
      ymin=-5,
      ymax=5,
      samples=100,
      samples y=0,
      hide axis,%
      ]%
      
      \addplot3[%
      domain=-pi/2:pi/2,%
      color = red,
      thick,
      ]%
      ({4.3*sin(deg(x))},{2.2*cos(deg(x))},{0});

      \addplot3[%
      domain=-1.5:1.5,%
      color = blue,
      thick,
      ]%
      ({2*cosh(x)},{0},{2*sinh(x)});

      \addplot3[%
      domain=-5:5,%
      ]%
      ({x},{-2.7},{0});%
      
      \addplot3[%
      domain=-5:5,%
      ]%
      ({x},{2.7},{0});%
      
      \addplot3[%
      domain=-2.7:2.7,%
      ]%
      ({-5},{x},{0});%
      
      \addplot3[%
      domain=-4.7:4.7,%
      ]%
      ({-5},{0},{x});%

      \addplot3[%
      domain=-5:5,%
      ]%
      ({x},{0},{-4.7});%
      
      \addplot3[%
      domain=-5:5,%
      ]%
      ({x},{0},{4.7});%
      
      \addplot3[%
      domain=-4.7:4.7,%
      ]%
      ({5},{0},{x});%
      
      \addplot3[%
      domain=-1.5:1.5,%
      color = blue,
      thick,
      ]%
      ({-2*cosh(x)},{0},{2*sinh(x)});

      \addplot3[%
      domain=pi/2:3*pi/2,%
      color = red,
      thick,
      ]%
      ({4.3*sin(deg(x))},{2.2*cos(deg(x))},{0});

      \addplot3[%
      domain=-5:5,%
      densely dotted,
      ]%
      ({x},{0},{0});%

      \addplot3[%
      domain=-2.7:2.7,%
      ]%
      ({5},{x},{0});%

      \addplot3[%
      color=blue,%
      mark = *,
      mark size = 1.2,
      draw=none,
      ]%
      table[row sep=crcr]%
      {%
        -2 0 0 \\%
        2 0 0 \\%
      };%

      \addplot3[%
      color=red,%
      mark = *,
      mark size = 1.2,
      draw=none,
      ]%
      table[row sep=crcr]%
      {%
        -4.3 0 0 \\%
        4.3 0 0 \\%
      };%
    \end{axis}
    \node[draw,rounded corners,inner sep=1mm] at (1.7,4.9)
    {$e \ne 1$};
  \end{tikzpicture}%
  \hspace{-15mm}%
  \begin{tikzpicture}
    \begin{axis}[%
      view={78}{20},%
      xlabel = x,
      ylabel = y,
      zlabel = z,
      ymin=-5,
      ymax=5,
      samples=100,
      samples y=0,
      hide axis,%
      ]%
      
      \addplot3[%
      domain=-2.7:2.7,%
      ]%
      ({-5},{x},{0});%

      \addplot3[%
      domain=-4.7:4.7,%
      ]%
      ({-5},{0},{x});%

      \addplot3[%
      domain=-2.4:2.4,%
      color = red,
      thick,
      ]%
      ({x^2-1.2},{x},{0});

      \addplot3[%
      domain=-2.4:2.4,%
      color = blue,
      thick,
      ]%
      ({1.2 - x^2},{0},{1.5*x});

      \addplot3[%
      domain=-5:5,%
      ]%
      ({x},{-2.7},{0});%
      
      \addplot3[%
      domain=-5:5,%
      ]%
      ({x},{2.7},{0});%
      
      \addplot3[%
      domain=-5:5,%
      ]%
      ({x},{0},{-4.7});%
      
      \addplot3[%
      domain=-5:5,%
      ]%
      ({x},{0},{4.7});%
      
      \addplot3[%
      domain=-4.7:4.7,%
      ]%
      ({5},{0},{x});%
      
      \addplot3[%
      domain=-5:5,%
      densely dotted,
      ]%
      ({x},{0},{0});%

      \addplot3[%
      domain=-2.7:2.7,%
      ]%
      ({5},{x},{0});%

      \addplot3[%
      color=blue,%
      mark = *,
      mark size = 1.2,
      draw=none,
      ]%
      table[row sep=crcr]%
      {%
        1.2 0 0 \\%
      };%

      \addplot3[%
      color=red,%
      mark = *,
      mark size = 1.2,
      draw=none,
      ]%
      table[row sep=crcr]%
      {%
        -1.2 0 0 \\%
      };%
    \end{axis}
    \node[draw,rounded corners,inner sep=1mm] at (1.7,4.9)
    {$e = 1$};
  \end{tikzpicture}%
  \vspace{-5mm}
  \caption{An illustration of \cref{tSolSat}, showing the quadric
    $\Sol$ of solutions and the quadric $\Sat$ of satellites (red and
    blue). Their roles are interchangeable. They live on perpendicular
    planes. The vertices of one quadric are the foci of the other and
    vice versa.} \label{fQuadrics}
\end{figure}
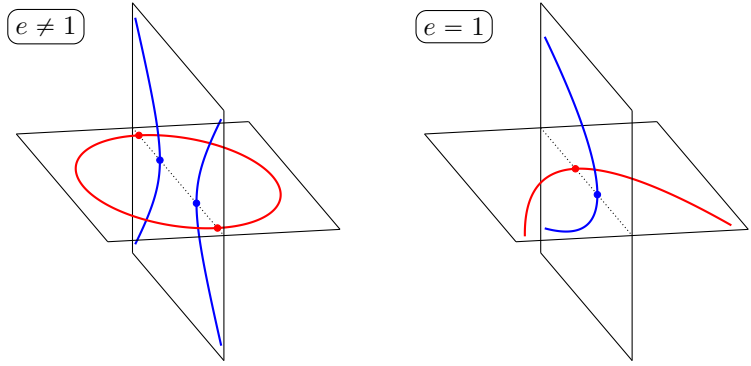

It may happen that the hypothesis $|\mathfrak X| > 1$ in 
\cref{tSat,tSolSat} is not satisfied. Then
$\Sat = \RR^n$ by \cref{rX1}, and the set $\mathcal S'$ in \cref{tSat}
does not satisfy $\mathcal S' = \mathcal S$. But that need not hinder
us from asking what $\mathcal S'$ and
\[
  \Sat' := \bigl\{\ve s \in \RR^n \mid (t,\ve s) \in \mathcal S' \
  \text{for some} \ t \in \RR\bigr\}
\]
look like. (In fact, the value of~$t$ has to be
$t = \langle\ve u,\ve s\rangle - \alpha$ by the definition of
$\mathcal S'$ in \cref{tSat}, so $\mathcal S'$ is determined by
$\Sat'$.) The definition of $\mathcal S'$ shows that $\Sat'$ consists
of all~$\ve s$ given by~\cref{eqSq} such that
$h(z_0 \upto z_{m-2}) = 0$, with~$h$ given by~\cref{eqH}. Moreover, if
$e \ne 1$, $h$ can be written as~\cref{eqH2}, but now~$\rho$ need not
satisfy the restriction $\rho > 0$ if $\sgn(e^2 - 1) = 1$. If $e = 1$,
$h$ can be written as~\cref{eqH3} or~\cref{eqCylinder}, where in the
latter case $\lVert\ve v\rVert^2 - \beta > 0$ is possible. So if we
drop the hypothesis $|\mathfrak X| > 1$, then $\Sat'$ could possibly
be a hyperboloid of revolution of one sheet (if $e > 1$ and
$\rho < 0$), a right circular cone (if $e > 1$ and $\rho = 0$) or a
right circular cylinder (if $e = 1$,
$\langle\ve u,\ve v\rangle = \alpha$ and
$\lVert\ve v\rVert^2 > \beta$), apart from the three types (prolate
spheroid, hyperbloid of revolution of two sheets, paraboloid of
revolution) we get if $|\mathfrak X| > 1$. Do all these types actually
occur? The following example shows that the answer is yes.

\begin{ex} \label{exQuadrics}%
  Take the points
  \[
    \ve s_1 = (-1,0,0), \ \ve s_2 = (1,0,0), \ \ve s_3 = (0,1,0), \
    \text{and}\ \ve s_4 = (3,0,4) \in \RR^3.
  \]
  \cref{taSat} shows the type of the quadric $\Sat'$ that arises for
  different choices of $t_1,t_2,t_3$, and~$t_4$.
  \begin{table}[hb]
    \newcommand{\entry}[1]{\parbox[c][9mm]{30mm}{\centering #1}}%
    \[
      \begin{array}{|c|c|c|c|}
        \hline
        \entry{$t_1,t_2,t_3,t_4$} & \entry{type of $\Sat'$} &%
        \entry{$|\mathfrak X|=$ number of solutions of~\cref{eqGPSsquared}}%
        & \entry{number of solutions of~\cref{eqGPS}} \\
        \hline
        \entry{$0,\sqrt{2},\sqrt{2}/2,4 \sqrt{2}$} & \entry{cylinder} & 
        \entry{0} & \entry{0} \\
        \hline
        \entry{$0,0,0,0$} & \entry{sphere} & \entry{2} & \entry{1} \\
        \hline
        \entry{$0,0,0,2$} & \entry{prolate spheroid} & \entry{2} &
        \entry{1} \\
        \hline
        \entry{$0,0,0,4$} & \entry{paraboloid} & \entry{1} &
        \entry{1} \\
        \hline
        \entry{$0,0,0,13/3$} & \entry{hyperboloid of two sheets} &
         \entry{2} & \entry{2} \\
        \hline
        \entry{$0,0,0,2 \sqrt{5}$} & \entry{circular cone} & \entry{1} &
        \entry{1} \\
        \hline
        \entry{$0,0,0,5$} & \entry{hyperboloid of one sheet} & \entry{0} &
        \entry{0} \\
        \hline
        \entry{$0,0,0,6$} & \entry{hyperboloid of two sheets} &
         \entry{2} & \entry{0} \\
         \hline
      \end{array}
    \]
    \caption{In \cref{exQuadrics}, the quadric $\Sat'$ can be of
      various types%
      .} \label{taSat}
  \end{table}
  All prognisticated types occur, including a sphere in the case that
  the~$t_i$ are all equal. Notice that under the hypotheses of
  \cref{tSolSat}, a paraboloid is also excluded since it can only
  occur if $m \le n$.
\end{ex}

\section{The global positioning inequalities} \label{sIneq}

In the previous \cref{sQsol,sQsat,sShape} we have only considered the
squared global positioning equations~\cref{eqGPSsquared}. In this
section we come back to the original ones~\cref{eqGPS}, which are
equivalent to~\cref{eqGPSsquared} together with the
inequalities~\cref{eqIneq}.

A look at \cref{taQuadrics} shows that $\Sat$ or $\Sol$, defined
in~\cref{eqQsol,eqQsat}, can be a hyperboloid of two sheets (for
shortness called ``hyperboloid'' in the table). The two sheets play a
role in \cref{tIneqA,tIneqB} of the following result. The sheets are
distinguished by the condition that the connecting vector
$\ve p - \ve c$ from the center of the hyperboloid to a point~$\ve p$
of it has positive or negative scalar product with the vector~$\ve u$,
which gives the direction of the axis of symmetry.


\begin{theorem}[The global positioning inequalities] \label{tIneq}%
  Let $\ve s_1 \upto \ve s_m \in \RR^n$ be in general linear position,
  $m \le n+1$, and let $t_1 \upto t_m \in \RR$ be numbers that are not
  all equal. Assume that~\cref{eqGPSsquared} has at least two
  solutions, and that $m > 2$ or
  $|t_1 - t_2| \ne \lVert\ve s_1 - \ve s_2\rVert$. Let
  $(b,\ve x) \in \mathfrak X$ and $(t,\ve s) \in \mathcal S$ (with
  $\mathfrak X$ and $\mathcal S$ defined in \cref{eqX,eqS}).
  
  \begin{enumerate}[label=(\alph*)]
  \item \label{tIneqA} If $e > 1$ we write $\Sat^+$ for the sheet of
    $\Sat$ that is contained in the half space
    $\bigl\{\ve p \in \RR^n \mid \langle\ve u,\ve p - \ve c\rangle \ge
    0\bigr\}$. Then
    \[
      \lVert\ve s - \ve x\rVert = t - b \qquad \Longleftrightarrow
      \qquad t \ge b \qquad \Longleftrightarrow \qquad \ve s \in
      \Sat^+.
    \]
    So~\cref{eqGPS} is solvable if and only if $\ve s_1 \upto \ve s_m$
    all lie on the sheet $\Sat^+$, and in this case the set of solutions
    of~\cref{eqGPS} is $\mathfrak X$, the same as the set of solutions
    of~\cref{eqGPSsquared}.
  \item \label{tIneqB} If $e < 1$ we write $\Sol^-$ for the sheet of
    $\Sol$ that is contained in the half space
    $\bigl\{\ve p \in \RR^n \mid \langle\ve u,\ve p - \ve c\rangle \le
    0\bigr\}$. Then
    \[
      \lVert\ve s - \ve x\rVert = t - b \qquad \Longleftrightarrow
      \qquad t \ge b \qquad \Longleftrightarrow \qquad \ve x \in
      \Sol^-.
    \]
    So the set of solutions of~\cref{eqGPS} consists of those
    solutions $(b,\ve x)$ of~\cref{eqGPSsquared} with
    $\ve x \in \Sol^-$.
  \item \label{tIneqC} Suppose $e = 1$. Then
    \[
      \lVert\ve s - \ve x\rVert = t - b \qquad \Longleftrightarrow
      \qquad t \ge b \qquad \Longleftrightarrow \qquad \langle\ve
      u,\ve v\rangle > \alpha.
    \]
    (The latter condition is independent of the choice of $(b,\ve x)$
    and $(t,\ve s)$.) So~\cref{eqGPS} is solvable if and only if
    $\langle\ve u,\ve v\rangle > \alpha$. If this is satisfied, then
    the set of solutions of~\cref{eqGPS} is $\mathfrak X$, the same as
    the set of solutions of~\cref{eqGPSsquared}. The condition
    $\langle\ve u,\ve v\rangle > \alpha$ means that the paraboloid
    $\Sat$ opens in the direction of the vector~$\ve u$.
  \end{enumerate}
\end{theorem}

\begin{proof}
  In all parts of the theorem, the first equivalence is clear since
  $(b,\ve x) \in \mathfrak X$ and $(t,\ve s) \in \mathcal S$ mean
  $\lVert\ve s - \ve x\rVert^2 = (t - b)^2$. So only the second
  equivalence needs proving.
  
  By \cref{tSol}, every $\ve x \in \Sol$ is of the form
  $\ve v + y_0 \ve u + \sum_{j=1}^k y_j \ve w_j$ and the~$b$ with
  $(b,\ve x) \in \mathfrak X$ is unique, so we can write $b(\ve x)$
  for it. Specifically, $b(\ve x) = y_0$. Likewise,
  \cref{tSolSat}\ref{tSolSatA} tells us that every $\ve s \in \Sat$
  can be written as
  $\ve s = \ve v + z_0 \ve u + \textstyle \sum_{i=1}^{m-2} z_i \ve
  w_i'$, and by \cref{tSat} the~$t$ with $(t,\ve s) \in \mathcal S$ is
  unique, and we write $t(\ve s)$ for it. In fact,
  \begin{equation} \label{eqTs}%
    t(\ve s) = \langle\ve u,\ve s\rangle - \alpha = \langle\ve u,\ve
    v\rangle - \alpha + e^2 z_0.    
  \end{equation}
  Before going into the cases addressed by
  \cref{tIneqA,tIneqB,tIneqC}, we look at the case $e \ne 1$. We pick
  out the vertices of $\Sol$ and $\Sat$ from \cref{taQuadrics}, which
  are
  $\ve x_\sigma = \ve v - \mu \ve u + \sigma \sqrt{\frac{\rho}{e^2 -
      1}} \ve u$ (with $\sigma = \pm 1$) for $\Sol$, and
  $\ve s_\tau = \ve v - \mu \ve u + \tau e^{-1} \sqrt{\frac{\rho}{e^2
      - 1}} \ve u$ (with $\tau = \pm 1$) for $\Sat$. So
  $b(\ve x_\sigma) = - \mu + \sigma \sqrt{\frac{\rho}{e^2 - 1}}$ and
  $t(\ve s_\tau) = \langle\ve u,\ve v\rangle - \alpha - e^2 \mu + \tau
  e \sqrt{\frac{\rho}{e^2 - 1}} = - \mu + \tau e \sqrt{\frac{\rho}{e^2
      - 1}}$ since
  $\langle\ve u,\ve v\rangle - \alpha = (e^2 - 1) \mu$
  by~\cref{eqEmd}. We obtain
  \begin{equation} \label{eqSign}
    \sgn\bigl(t(\ve s_\tau) - b(\ve x_\sigma)\bigr) = \sgn(e \tau -
    \sigma) = \left\{
    \begin{array}{cl}
      \tau & \text{if} \ e > 1 \\
      - \sigma & \text{if} \ e < 1
    \end{array}\right..    
  \end{equation}
  \begin{enumerate}
  \item[\ref{tIneqA}] Now assume $e > 1$. We are given
    $(b,\ve x) \in \mathfrak X$ and $(t,\ve s) \in \mathcal S$, so
    $b = b(\ve x)$ and $t = t(\ve s)$. Moreover, $\ve x \in \Sol$ and
    $\ve s \in \Sat$, so by \cref{tSolSat} the corresponding~$y_i$
    and~$z_i$ satisfy $g(y_0 \upto y_k) = 0$ and
    $h(z_0 \upto z_{m-2}) = 0$. A look at~\cref{eqG,eqH2} shows
    \[
      b(\ve x_{-1}) \le b(\ve x) \le b(\ve x_{+1}) \quad \text{and}
      \quad \left\{
        \begin{array}{cl}
          t(\ve s) \ge t(\ve s_{+1}) & \text{if} \ \ve s \in \Sat^+ \\
          t(\ve s) \le t(\ve s_{-1}) & \text{if} \ \ve s \in \Sat^-
        \end{array}
      \right\}.
    \]
    So if $\ve s \in \Sat^+$ then
    $t = t(\ve s) \ge t(\ve s_{+1}) \underset{\cref{eqSign}}{>} b(\ve
    x_{+1}) \ge b(\ve x) = b$, and if $\ve s \in \Sat^-$ then
    $t = t(\ve s) \le t(\ve s_{-1}) \underset{\cref{eqSign}}{<} b(\ve
    x_{-1}) \le b(\ve x) = b$.
  \item[\ref{tIneqB}] For $e < 1$ we see from ~\cref{eqG,eqH2} that
    \[
      \left\{
        \begin{array}{cl}
          b(\ve x) \ge b(\ve x_{+1}) & \text{if} \ \ve x \in \Sol^+ \\
          b(\ve x) \le b(\ve x_{-1}) & \text{if} \ \ve x \in \Sol^-
        \end{array}
      \right\} \quad \text{and} \quad t(\ve s_{-1}) \le t(\ve s) \le
      t(\ve s_{+1}).
    \]
    So if $\ve x \in \Sol^-$ then
    $t = t(\ve s) \ge t(\ve s_{-1}) \underset{\cref{eqSign}}{>} b(\ve
    x_{-1}) \ge b(\ve x) = b$, and if $\ve x \in \Sol^+$ then
    $t = t(\ve s) \le t(\ve s_{+1}) \underset{\cref{eqSign}}{<} b(\ve
    x_{+1}) \le b(\ve x) = b$.
  \item[\ref{tIneqC}] In the case $e = 1$, $\Sol$ and $\Sat$ are both
    paraboloids with vertices $\ve x_0 := \ve v + \lambda_1 \ve u$ and
    $\ve s_0 := \ve v + \lambda_2 \ve u$ according to
    \cref{taQuadrics}, with the $\lambda_i$ given
    by~\cref{eqLambda}. We have
    \begin{equation} \label{eqT0}%
      t(\ve s_0) - b(\ve x_0) \underset{\cref{eqTs}}{=} \langle\ve
      u,\ve s\rangle - \alpha + \lambda_2 - \lambda_1 =
      \frac{\langle\ve u,\ve s\rangle - \alpha}{2}.
    \end{equation}
    We are given $(b,\ve x) \in \mathfrak X$ and
    $(t,\ve s) \in \mathcal S$, so $b = b(\ve x)$ and $t = t(\ve
    s)$. Moreover, $\ve x \in \Sol$ and $\ve s \in \Sat$, so by
    \cref{tSolSat} the corresponding~$y_i$ and~$z_i$ satisfy
    $g(y_0 \upto y_k) = 0$ and $h(z_0 \upto z_{m-2}) = 0$, with~$g$
    and~$h$ now given by~\cref{eqG2,eqH3}. So if
    $\langle\ve u,\ve v\rangle > \alpha$ then
    \[
      t = t(\ve s) \underset{\cref{eqH3}}{\ge} t(\ve s_0)
      \underset{\cref{eqT0}}{>} b(\ve x_0) \underset{\cref{eqG2}}{\ge}
      b(\ve x) = b,
    \]
    and if $\langle\ve u,\ve v\rangle < \alpha$ then
    \[
      t = t(\ve s) \underset{\cref{eqH3}}{\le} t(\ve s_0)
      \underset{\cref{eqT0}}{<} b(\ve x_0) \underset{\cref{eqG2}}{\le}
      b(\ve x) = b.
    \]
    The geometric interpretation of the condition
    $\langle\ve u,\ve v\rangle > \alpha$ follows
    from~\cref{eqH3}. \endproof
  \end{enumerate}
\end{proof}

The theorem is illustrated in \cref{fIneq}, which, as
\cref{fQuadrics}, shows the case $n = m = 3$, so $k = 1$. 

\begin{figure}[!h]
  \usetikzlibrary{arrows.meta}
  \definecolor{forestgreen}{rgb}{0.13, 0.55, 0.13}
  \begin{tikzpicture}
    \begin{axis}[%
      view={78}{20},%
      xlabel = x,
      ylabel = y,
      zlabel = z,
      ymin=-5,
      ymax=5,
      samples=100,
      samples y=0,
      hide axis,%
      ]%
      \addplot3[%
      domain=-5:5,%
      densely dotted,
      ]%
      ({x},{0},{0});%
      
      \coordinate (tail) at (axis cs:-3.2,0,0);
      \coordinate (head) at (axis cs:1.3,0,0);
    \end{axis}
    
    \draw[-{Stealth[scale=1.7]},forestgreen] (tail)--(head);

    \begin{axis}[%
      view={78}{20},%
      xlabel = x,
      ylabel = y,
      zlabel = z,
      ymin=-5,
      ymax=5,
      samples=100,
      samples y=0,
      hide axis,%
      ]%
      
      \addplot3[%
      domain=-5:5,%
      ]%
      ({x},{-2.7},{0});%
      
      \addplot3[%
      domain=-5:5,%
      ]%
      ({x},{2.7},{0});%
      
      \addplot3[%
      domain=-2.7:2.7,%
      ]%
      ({-5},{x},{0});%
      
      \addplot3[%
      domain=-4.7:4.7,%
      ]%
      ({-5},{0},{x});%

      \addplot3[%
      domain=-5:5,%
      ]%
      ({x},{0},{-4.7});%
      
      \addplot3[%
      domain=-5:5,%
      ]%
      ({x},{0},{4.7});%
      
     \addplot3[%
      domain=pi/2:3*pi/2,%
      color = red,
      thick,
      ]%
      ({4.3*sin(deg(x))},{0},{3.2*cos(deg(x))});

      \addplot3[%
      domain=-1.5:1.5,%
      color = blue,
      thick,
      ]%
      ({2*cosh(x)},{1.1*sinh(x)},{0});

      \addplot3[%
      domain=-pi/2:pi/2,%
      color = red,
      thick,
      ]%
      ({4.3*sin(deg(x))},{0},{3.2*cos(deg(x))});

      \addplot3[%
      domain=-2.7:2.7,%
      ]%
      ({5},{x},{0});%

      \addplot3[%
      domain=-4.7:4.7,%
      ]%
      ({5},{0},{x});%
      
      \addplot3[%
      color=blue,%
      mark = *,
      mark size = 1.2,
      draw=none,
      ]%
      table[row sep=crcr]%
      {%
        2 0 0 \\%
      };%

      \addplot3[%
      color=red,%
      mark = *,
      mark size = 1.2,
      draw=none,
      ]%
      table[row sep=crcr]%
      {%
        -4.3 0 0 \\%
        4.3 0 0 \\%
      };%
    \end{axis}
    \node[draw,rounded corners,inner sep=1mm] at (1.7,4.9)
    {$e > 1$};
  \end{tikzpicture}%
  \hspace{-30mm}%
  \begin{tikzpicture}
    \begin{axis}[%
      view={78}{20},%
      xlabel = x,
      ylabel = y,
      zlabel = z,
      ymin=-5,
      ymax=5,
      samples=100,
      samples y=0,
      hide axis,%
      ]%
      \addplot3[%
      domain=-5:5,%
      densely dotted,
      ]%
      ({x},{0},{0});%
      
      \coordinate (tail) at (axis cs:-1.3,0,0);
      \coordinate (head) at (axis cs:3.2,0,0);
    \end{axis}
    
    \draw[-{Stealth[scale=1.7]},forestgreen] (tail)--(head);

    \begin{axis}[%
      view={78}{20},%
      xlabel = x,
      ylabel = y,
      zlabel = z,
      ymin=-5,
      ymax=5,
      samples=100,
      samples y=0,
      hide axis,%
      ]%
      
      \addplot3[%
      domain=-pi/2:pi/2,%
      color = blue,
      thick,
      ]%
      ({4.3*sin(deg(x))},{2.2*cos(deg(x))},{0});


      \addplot3[%
      domain=-5:5,%
      ]%
      ({x},{-2.7},{0});%
      
      \addplot3[%
      domain=-5:5,%
      ]%
      ({x},{2.7},{0});%
      
      \addplot3[%
      domain=-2.7:2.7,%
      ]%
      ({-5},{x},{0});%
      
      \addplot3[%
      domain=-4.7:4.7,%
      ]%
      ({-5},{0},{x});%

      \addplot3[%
      domain=-5:5,%
      ]%
      ({x},{0},{-4.7});%
      
      \addplot3[%
      domain=-5:5,%
      ]%
      ({x},{0},{4.7});%
      
      \addplot3[%
      domain=-4.7:4.7,%
      ]%
      ({5},{0},{x});%
      
      \addplot3[%
      domain=-1.5:1.5,%
      color = red,
      thick,
      ]%
      ({-2*cosh(x)},{0},{2*sinh(x)});

      \addplot3[%
      domain=pi/2:3*pi/2,%
      color = blue,
      thick,
      ]%
      ({4.3*sin(deg(x))},{2.2*cos(deg(x))},{0});

      \addplot3[%
      domain=-2.7:2.7,%
      ]%
      ({5},{x},{0});%

      \addplot3[%
      color=red,%
      mark = *,
      mark size = 1.2,
      draw=none,
      ]%
      table[row sep=crcr]%
      {%
        -2 0 0 \\%
      };%

      \addplot3[%
      color=blue,%
      mark = *,
      mark size = 1.2,
      draw=none,
      ]%
      table[row sep=crcr]%
      {%
        -4.3 0 0 \\%
        4.3 0 0 \\%
      };%
    \end{axis}
    \node[draw,rounded corners,inner sep=1mm] at (1.7,4.9)
    {$e < 1$};
  \end{tikzpicture}%
  \hspace{-30mm}%
  \begin{tikzpicture}
    \begin{axis}[%
      view={78}{20},%
      xlabel = x,
      ylabel = y,
      zlabel = z,
      ymin=-5,
      ymax=5,
      samples=100,
      samples y=0,
      hide axis,%
      ]%
      \addplot3[%
      domain=-5:5,%
      densely dotted,
      ]%
      ({x},{0},{0});%
      
      \coordinate (tail) at (axis cs:-3.7,0,0);
      \coordinate (head) at (axis cs:3.7,0,0);
    \end{axis}
    
    \draw[-{Stealth[scale=1.7]},forestgreen] (tail)--(head);

    \begin{axis}[%
      view={78}{20},%
      xlabel = x,
      ylabel = y,
      zlabel = z,
      ymin=-5,
      ymax=5,
      samples=100,
      samples y=0,
      hide axis,%
      ]%
      
      \addplot3[%
      domain=-2.7:2.7,%
      ]%
      ({-5},{x},{0});%

      \addplot3[%
      domain=-4.7:4.7,%
      ]%
      ({-5},{0},{x});%

      \addplot3[%
      domain=-2.4:2.4,%
      color = blue,
      thick,
      ]%
      ({x^2-1.2},{x},{0});

      \addplot3[%
      domain=-2.4:2.4,%
      color = red,
      thick,
      ]%
      ({1.2 - x^2},{0},{1.5*x});

      \addplot3[%
      domain=-5:5,%
      ]%
      ({x},{-2.7},{0});%
      
      \addplot3[%
      domain=-5:5,%
      ]%
      ({x},{2.7},{0});%
      
      \addplot3[%
      domain=-5:5,%
      ]%
      ({x},{0},{-4.7});%
      
      \addplot3[%
      domain=-5:5,%
      ]%
      ({x},{0},{4.7});%
      
      \addplot3[%
      domain=-4.7:4.7,%
      ]%
      ({5},{0},{x});%
      
      \addplot3[%
      domain=-2.7:2.7,%
      ]%
      ({5},{x},{0});%

      \addplot3[%
      color=red,%
      mark = *,
      mark size = 1.2,
      draw=none,
      ]%
      table[row sep=crcr]%
      {%
        1.2 0 0 \\%
      };%

      \addplot3[%
      color=blue,%
      mark = *,
      mark size = 1.2,
      draw=none,
      ]%
      table[row sep=crcr]%
      {%
        -1.2 0 0 \\%
      };%
    \end{axis}
    \node[draw,rounded corners,inner sep=1mm] at (1.7,4.9)
    {$e = 1$};
  \end{tikzpicture}%
  \vspace{-5mm}
  \caption{An illustration of \cref{tIneq}, showing the locus of
    satellites (blue) and the set of solutions (red) of the
    system~\cref{eqGPS}. The green arrow indicates the direction of
    the vector~$\ve u$. If $e > 1$, then all satellites need to be on
    one sheet of $\Sat$. If $e < 1$, then the solutions
    of~\cref{eqGPS} are on one sheet of $\Sol$.} \label{fIneq}
\end{figure}
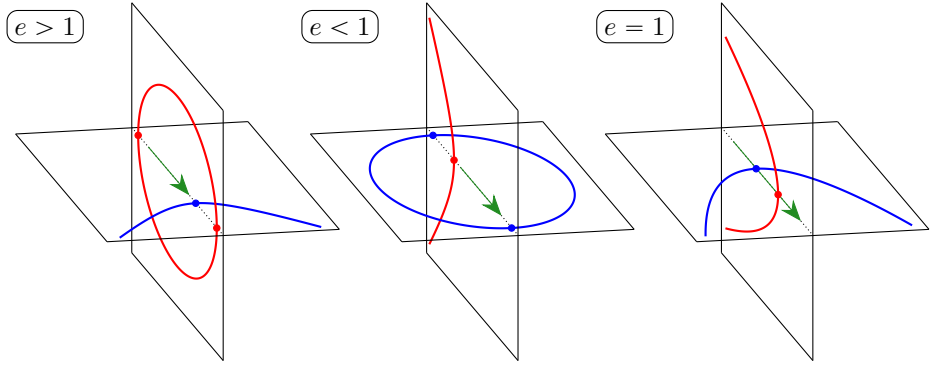

The following two examples treat the extreme cases of $m = 2$ and
$m = n + 1$ satellites in view.

\begin{ex}[Two satellites in view] \label{exM2}%
  Let us consider the case $m = 2$. So we are given two distinct
  points $\ve s_1,\ve s_2 \in \RR^n$ and distinct numbers (the
  pseudoranges) $t_1$ and~$t_2$. \cref{lM2} tells us
  that~\cref{eqGPSsquared} has at least two solutions, and
  $e = \frac{|t_1 - t_2|}{\lVert\ve s_1 - \ve s_2\rVert}$. So the case
  $e = 1$ is not covered by \cref{tSolSat,tIneq}. (In fact, $e = 1$ is
  the case of \cref{pCollinear}.) Let us therefore assume $e \ne
  1$. Then by \cref{tSolSat}\ref{tSolSatC}, the quadric $\Sat$
  consists of only two points, and with \cref{tSat} it follows that
  \[
    \Sat = \{\ve s_1,\ve s_2\}.
  \]
  Moreover, \cref{tSolSat} says that $\Sol$ is a quadric in
  $\Asol = \RR^n$ (since $k + 1 = n + 2 - m = n$) with eccentricity
  $e^{-1}$ and foci~$\ve s_1$ and~$\ve s_2$.

  If $e > 1$, then $\Sol$ is a spheroid. Moreover, one of the
  points~$\ve s_i$ lies in $\Sat^+$ and the other in $\Sat^-$. So even
  though all points $\ve x \in \Sol$ come from a solution
  of~\cref{eqGPSsquared}, \cref{tIneq}\ref{tIneqA} says
  that~\cref{eqGPS} has no solution at all. This makes perfect sense,
  since $e > 1$ means that the difference $|t_1 - t_2|$ between the
  pseudoranges, and hence also the difference between the running
  times of the signals, is bigger than the distance between the
  satellites.

  On the other hand, if $e < 1$, then $\Sol$ is a hyperboloid of
  revolution of two sheets, and \cref{tIneq}\ref{tIneqB} tells us that
  the points of one of the sheets, $\Sol^-$, are precisely those that
  come from solutions of~\cref{eqGPS}. Thus we have recovered, as a
  special case, the well-known fact that the solutions of the global
  positioning problem for two satellites lie on one sheet of a
  hyperboloid of revolution with the satellite positions as foci. This
  is the origin of the the term ``hyperbolic navigation'' (see, for
  example, \mycite{Lee:1975}).
\end{ex}


\begin{ex}[The case $m = n + 1$] \label{exMn1}%
  If $m = n + 1$ then $k = 0$, so by \cref{tSolSat}\ref{tSolSatC},
  $\Sol$ consists of two points, so we can write
  $\Sol = \{\ve x,\ve x'\}$. \cref{tSolSat} also tells us that $\Sat$
  is a quadric in $\Asat = \RR^n$ with foci~$\ve x$ and~$\ve
  x'$. Looking at \cref{taQuadrics} we see that $m = n + 1$ implies
  $e \ne 1$.

  If $e > 1$, then $\Sat$ is a hyperboloid of revolution of two
  sheets, $\Sat^+$ and $\Sat^-$. \cref{tIneq}\ref{tIneqA} says that
  the global positioning problem~\cref{eqGPS} is solvable if and only
  if all satellite positions $\ve s_i$ are on $\Sat^+$. In this case,
  both~$\ve x$ and~$\ve x'$ are solutions. \cref{taSat} contains a
  numerical example where this happens (the fourth row from the
  bottom), and another one where not all satellites are on $\Sat^+$
  (the bottom row).

  If, on the other hand, $e < 1$, then $\Sat$ is a prolate spheroid,
  and from \cref{tIneq}\ref{tIneqB} we gather that~\cref{eqGPS} has a
  unique solution. In fact, the two sheets $\Sol^+$ and $\Sol^-$ are
  reduced to single points ($\ve x$ and~$\ve x'$), and the solution is
  the one that makes up $\Sol^-$. The vector~$\ve u$ distinguishes
  them: it points from the point which solves~\cref{eqGPS} to the one
  that only solves~\cref{eqGPSsquared}.

  So as the special case $m = n + 1$, we have recovered one of the
  main results 
  from~[\citenumber{Boutin:Kemper:2024}].
\end{ex}

\section{The case of different ranks} \label{sRank}%

Recall from \cref{sSpheres} that we are given satellite positions
$\ve s_1 \upto \ve s_m \in \RR^n$ and pseudoranges
$t_1 \upto t_m \in \RR$, and we wish to solve the squared global
positioning equations~\cref{eqGPSsquared} with unknowns
$\ve x \in \RR^n$ and $b \in \RR$. For the matrices
$A \in \RR^{m \times (n+2)}$ and $B \in \RR^{m \times (n+1)}$ defined
in~\cref{eqAGPS,eqB} we have until now assumed that $\rank(B) = m$,
which implies $\rank(A) = \rank(B) = m$. In this section we will
assume instead that
\begin{equation} \label{eqRank}%
  \rank(A) = m \ \text{but} \ \rank(B) < m, \ \text{so} \ \rank(B) = m
  - 1.
\end{equation}
We will see in \cref{sDuality} that this assumption or the one made in
\cref{sQsol,sQsat,sShape} provides sufficient generality. Our
assumption implies that every linear system with coefficient matrix
$A$ is solvable, and a solution has a unique value for the first
unknown: swap the first column to the rightmost position and then
perform Gaussian elimination. As explained in ~\cref{sSpheres}, for
solving~\cref{eqGPSsquared} we need to consider the
system~\cref{eqLinear} with
$\langle\vh s_i,\vh s_i\rangle_Q - d_i = \lVert\ve s_i\rVert^2 -
t_i^2$. So we can pick a particular solution, which we write as
$\left(\begin{smallmatrix} \vh x \\ \lambda\end{smallmatrix}\right) =
\left(\begin{smallmatrix} b_0 \\ \ve v' \\ \lambda
  \end{smallmatrix}\right)$ with $\ve v' \in \RR^n$ and $\lambda \in
\RR$. As just noted, $b_0$ is uniquely determined. This means
that~$b_0$ is the unique value for the bias~$b$ in a solution
of~\cref{eqGPSsquared}. We have
\[
  A \cdot
  \begin{pmatrix}
    b_0 \\ \ve v' \\ \lambda
  \end{pmatrix} =
  \begin{pmatrix}
    \lVert\ve s_1\rVert^2 - t_1^2 \\
    \vdots \\
    \lVert\ve s_m\rVert^2 - t_m^2
  \end{pmatrix},
\]
which, recalling the definition of $A$ and $B$, we can rewrite as
\[
  B \cdot
  \begin{pmatrix}
    \ve v' \\ \lambda + b_0^2
  \end{pmatrix} =
  \begin{pmatrix}
    \lVert\ve s_1\rVert^2 - (t_1 - b_0)^2 \\
    \vdots \\
    \lVert\ve s_m\rVert^2 - (t_m - b_0)^2
  \end{pmatrix}.    
\]
Consider the space
$U := \langle\ve s_i - \ve s_1 \mid 2 \le i \le m\rangle_\mathrm{lin}
\subseteq \RR^n$, which by the hypothesis $\rank(B) = m - 1$ has
dimension equal to $m - 2$. As in \cref{sQsol}, choose an orthonormal
basis $\ve w_1 \upto \ve w_k$ of the orthogonal complement $U^\perp$,
so here $k = n - (m - 2)$. With
$\gamma_j := \langle\ve s_i,\ve w_j\rangle$, which does not depend
on~$i$, the vectors
$\left(\begin{smallmatrix} \ve w_j \\ 2
    \gamma_j\end{smallmatrix}\right)$ form a basis of the kernel of
$B$. Thus by subtracting a suitable linear combination of the
$\left(\begin{smallmatrix} \ve w_j \\ 2
    \gamma_j\end{smallmatrix}\right)$ from
$\left(\begin{smallmatrix} \ve v' \\ \lambda +
    b_0^2\end{smallmatrix}\right)$, we obtain a (uniquely determined)
vector $\left(\begin{smallmatrix} \ve v \\
    \beta\end{smallmatrix}\right)$ such that
\begin{equation} \label{eqVbeta2}%
  B \cdot
  \begin{pmatrix}
    \ve v \\ \beta
  \end{pmatrix} =
  \begin{pmatrix}
    \lVert\ve s_1\rVert^2 - (t_1 - b_0)^2 \\
    \vdots \\
    \lVert\ve s_m\rVert^2 - (t_m - b_0)^2
  \end{pmatrix}    
\end{equation}
and $\langle\ve v,\ve w_j\rangle = \gamma_j$ ($j = 1 \upto k$).

\begin{theorem}[The set of solutions] \label{tSol2}%
  In the above setting the set of solutions of~\cref{eqGPSsquared} is
  \begin{multline*}
    \mathfrak X = \Bigl\{(b_0,\ve x) \ \Bigl| \ \ve x = \ve v +
    \sum_{j=1}^k y_j \ve w_j \ \text{and} \ \sum_{j=1}^k y_j^2 = \beta
    - \lVert\ve v\rVert^2\Bigr\} \\
    = \Bigl\{(b_0,\ve v + \ve y) \ \Bigl| \ \ve y \in U^\perp \
    \text{and} \ \lVert\ve y\rVert^2 = \beta - \lVert\ve
    v\rVert^2\Bigr\}.
  \end{multline*}
  The point~$\ve v$ lies in the affine span
  $\langle\ve s_1 \upto \ve s_m\rangle_\mathrm{aff}$.
\end{theorem}

\begin{proof}
  We have already seen that~$b_0$ is the unique value for the bias~$b$
  in every solution of~\cref{eqGPSsquared}. So we are left with the
  system
  \[
    \lVert\ve s_i - \ve x\rVert^2 = (t_i - b_0)^2 \qquad (i = 1 \upto
    m)
  \]
  for the unkown point~$\ve x$. To this system we apply the method for
  the sphere intersection problem from \cref{sSpheres}, so now
  $Q = I_n$ and $d_i = (t_i - b_0)^2$. Thus the linear
  system~\cref{eqLinear} becomes the system with matrix $B$ and right
  hand side as in~\cref{eqVbeta2}. The solution space is
  \[
    L = \Bigl\{
    \begin{pmatrix}
      \ve v \\ \beta
    \end{pmatrix} + \sum_{j=1}^k y_j
    \begin{pmatrix}
      \ve w_j \\ 2 \gamma_j
    \end{pmatrix} \ \Bigl| \ y_1 \upto y_k \in \RR\Bigr\}.
  \]
  So we obtain $\mathfrak X$ by imposing the
  equation~\cref{eqQuadratic} on the points of $L$, which leads to
  \[
    0 = \bigl\lVert\ve v + {\textstyle\sum_{j=1}^k y_j \ve
      w_j}\bigr\rVert^2 - \beta - 2 \sum_{j=1}^k \gamma_j y_j =
    \sum_{j=1}^k y_j^2 + \lVert\ve v\rVert^2 - \beta.
  \]
  So we get the formulas from the theorem. The last claim follows from
  $\langle\ve v,\ve w_j\rangle = \gamma_j$.
\end{proof}

So, depending on~$k$ (which may be any integer between~$0$ and~$n$)
and on $\beta - \lVert\ve v\rVert^2$, $\Sol$, the quadric of
solutions, can be empty, a single point, or a sphere around~$\ve v$
that lives in the affine subspace $\ve v + U^\perp$. Notice that by
the definition of $U$, we have
$\langle\ve s_1 \upto \ve s_m\rangle_\mathrm{aff} = \ve v + U$.

The next step is to determine the set $\mathcal S$ as defined
in~\cref{eqS}.

\begin{theorem}[The locus of satellites] \label{tSat2}%
  Assume the situation introduced before \cref{tSol2} and set
  \[
    \mathcal S' := \Bigl\{(t,\ve s) \in \RR \times \RR^n \ \Bigl| \
    \langle\ve s,\ve w_j\rangle = \gamma_j \ (j = 1 \upto k), \ (t -
    b_0)^2 = \lVert\ve s - \ve v\rVert^2 + \beta - \lVert\ve
    v\rVert^2\Bigr\}.
  \]
  Notice that the condition
  ``$\langle\ve s,\ve w_j\rangle = \gamma_j$'' means that
  $\ve s \in \langle\ve s_1 \upto \ve s_m\rangle_\mathrm{aff} = \ve v
  + U$.
  \begin{enumerate}[label=(\alph*)]
  \item \label{tSat2A} For $i = 1 \upto m$ we have
    $(t_i,\ve s_i) \in \mathcal S'$.
  \item \label{tSat2B} $\mathcal S' \subseteq \mathcal S$.
  \item \label{tSat2C} If $|\mathfrak X| > 1$, then
    $\mathcal S = \mathcal S'$.
  \end{enumerate}
\end{theorem}

The following lemma is analogous to \cref{lSat}.

\begin{lemma} \label{lSat2}%
  In the situation of \cref{tSol2}, let
  $(t,\ve s) \in \RR \times \RR^n$. Form the matrix
  $B'\in \RR^{(m+1) \times (n+1)}$ by appending the row
  $(2 \ve s^T,-1)$ at the bottom of $B$ as defined by~\cref{eqB}.  Let
  $\mathfrak X'$ be the set of solutions $(b,\ve x)$ of the
  system~\cref{eqGPSsquared}, enlarged by adding the pair
  $(\ve s_{m+1},t_{m+1}) := (\ve s,t)$. If $rank(B') = m - 1$, then
  \[
    \mathfrak X' = \left\{
      \begin{array}{ll}
        \mathfrak X & \text{if} \quad (t -
                      b_0)^2 = \lVert\ve s - \ve v\rVert^2 + \beta -
                      \lVert\ve v\rVert^2 \\
        \emptyset & \text{otherwise}
      \end{array}\right..
  \]
\end{lemma}

\begin{proof}
  Since they have the same rank, $B$ and $B'$ also have the same
  kernel. This implies $\langle\ve s,\ve w_j\rangle = \gamma_j$ for
  all~$j$, so $\ve s - \ve v \in U$. First assume that
  $(t - b_0)^2 = \lVert\ve s - \ve v\rVert^2 + \beta - \lVert\ve
  v\rVert^2$. Let $(b,\ve x) \in \mathfrak X$. Then by \cref{tSol2} we
  have $b = b_0$ and $\ve x = \ve v + \ve y$ with $\ve y \in U^\perp$,
  $\lVert\ve y\rVert^2 = \beta - \lVert\ve v\rVert^2$. Therefore
  \[
    \lVert\ve s - \ve x\rVert^2 = \lVert\ve s - \ve v - \ve y\rVert^2
    = \lVert\ve s - \ve v\Vert^2 + \lVert\ve y\rVert^2 = (t - b_0)^2 -
    \beta + \lVert\ve v\rVert^2 + \beta - \lVert\ve v\rVert^2 = (t -
    b_0)^2.
  \]
  So $(b,\ve x) \in \mathfrak X'$, and we get
  $\mathfrak X' \subseteq \mathfrak X$. The reverse inclusion is
  clear, so $\mathfrak X' = \mathfrak X$.

  Now assume that $\mathfrak X' \ne \emptyset$ and choose
  $(b,\ve x) \in \mathfrak X'$. Then also $(b,\ve x) \in \mathfrak X$,
  so as above $b = b_0$ and $\ve x = \ve v + \ve y$ with
  $\ve y \in \U^\perp$,
  $\lVert\ve y\rVert^2 = \beta - \lVert\ve v\rVert^2$. Since $(b,\ve
  x) \in \mathfrak X'$ we obtain
  \[
    0 = \lVert\ve s - \ve x\rVert^2 - (t - b_0)^2 = \lVert\ve s - \ve
    v + \ve y\rVert^2 - (t - b_0)^2 = \lVert\ve s - \ve v\Vert^2 +
    \beta - \lVert\ve v\rVert^2 - (t - b_0)^2.
  \]
  So if the equation
  $(t - b_0)^2 = \lVert\ve s - \ve v\rVert^2 + \beta - \lVert\ve
  v\rVert^2$ is not satisfied then $\mathfrak X' = \emptyset$.
\end{proof}

\begin{proof}[Proof of \cref{tSat2}]
  \begin{enumerate}
  \item[\ref{tSat2A}] For all~$i$ and~$j$ we have
    $\langle\ve s_i,\ve w_j\rangle = \gamma_j$ by the definition of
    the~$\gamma_j$. Moreover, from~\cref{eqVbeta2} we have
    $2 \langle\ve s_i,\ve v\rangle - \beta = \lVert\ve s_i\rVert^2 -
    (t_i - b_0)^2$, so
    $(t_i - b_0)^2 = \lVert\ve s_i - \ve v\rVert^2 + \beta - \lVert\ve
    v\rVert^2$. Thus $(t_i,\ve s_i) \in \mathcal S'$.
  \item[\ref{tSat2B}] Let $(t,\ve s) \in \mathcal S'$ and form the
    matrix $B' \in \RR^{(m+1) \times (n+1)}$ as in \cref{lSat2}. Since
    $\langle\ve s,\ve w_j\rangle = \gamma_j$, $B'$ and $B$ have the
    same kernel, so $\rank(B') = m - 1$. So \cref{lSat2} tells us that
    $\mathfrak X' = \mathfrak X$ and thus $(t,\ve s) \in \mathcal S$.
  \item[\ref{tSat2C}] Since $|\mathfrak X| > 1$, \cref{tSol2} tells us
    that $\Sol$ is a sphere of positive radius around the
    point~$\ve v$ within the affine subspace $\ve v + U^\perp$. This
    implies~$k = \dim(U^\perp) > 0$. So we can pick
    $\ve x_1 \upto \ve x_{k+1} \in \Sol$ that are in general linear
    position, so the $\ve x_i - \ve x_1$ generate $U^\perp$.

    Now let $(t,\ve s) \in \mathcal S$. Then, since
    $(b_0,\ve x_i) \in \mathfrak X$ we obtain
    $\lVert\ve x_i - \ve s\rVert^2 - (b_0 - t)^2 = 0$ for
    $i = 1 \upto k+1$. So $(t,\ve s)$ solves the
    system~\cref{eqGPSsquared}, but with the $(b_0,\ve x_i)$ taking
    the role of known quantities. Therefore we are in the situation of
    \cref{exSphere1}. In this example, \cref{eqSolSphere} tells us
    that the solution point~$\ve s$ lies in the affine space
    containing the point~$\ve v$ whose associated linear space is the
    orthogonal complement of the space generated by the
    $\ve x_i - \ve x_1$. This means that $\ve s \in \ve v +
    U$. Therefore $\langle\ve s,\ve w_j\rangle = \gamma_j$, and we see
    again that the matrix $B'$ from \cref{lSat2} has rank $m - 1$. Now
    $(t,\ve s) \in \mathcal S$ implies
    $\mathfrak X' = \mathfrak X \ne \emptyset$, so the lemma says that
    $(t - b_0)^2 = \lVert\ve s - \ve v\rVert^2 + \beta - \lVert\ve
    v\rVert^2$. So $(t,\ve s)$ satisfies both conditions for lying in
    $\mathcal S'$, and the equality $\mathcal S = \mathcal S'$
    follows. \endproof
  \end{enumerate}
\end{proof}

\begin{rem} \label{rIneq}%
  What can we say about the inequalities~\cref{eqIneq} in the case of \cref{tSat2}? By \cref{tSat2A} and the formula for $\mathcal S'$ in the theorem, every $t_i - b_0$ equals $\pm \sqrt{\lVert\ve s_i - \ve v\rVert^2 + \beta - \lVert\ve
    v\rVert^2}$. So~\cref{eqIneq} is satisfied by the solutions in $\mathfrak X$ if and only if the sign in front of the square root is ``$+$'' for all~$i$. So, unlike in \cref{tIneq}, the solvability of~\cref{eqGPS} does not depend on the position of~$\ve x$ and the~$\ve s_i$, but on the values of~$t_i$.
\end{rem}

We now see, especially after having used \cref{exSphere1} in the above
proof, that the situations considered in this section and in
\cref{exSphere1} are dual to each other. Consequently, the following
theorem summarizes our results about both these situations. The
theorem complements \cref{tSolSat}.

\begin{theorem}[The quadric of solutions and the quadric of
  satellites] \label{tSolSat2}%
  Let $\ve s_1 \upto \ve s_m \in \RR^n$ and let
  $t_1 \upto t_m \in \RR$ such that the matrix
  $A \in \RR^{m \times (n+2)}$, given by~\cref{eqAGPS}, has
  rank~$m \ge 2$. Also assume that the system~\cref{eqGPSsquared} has
  at least two solutions. With $\Sol$ and $\Sat$ defined
  by~\cref{eqQsol,eqQsat}, we have:
  \begin{enumerate}[label=(\alph*)]
  \item \label{tSolSat2A} If $\ve s_1 \upto \ve s_m$ are {\em not} in
    general linear position, then
    \begin{enumerate}[label=(\arabic*)]
    \item \label{tSolSat2A1} $\Sat$ is the affine subspace of
      dimension $m - 2 < n$ given by
      $\Sat = \langle\ve s_1 \upto \ve s_m\rangle_\mathrm{aff}$.
    \item \label{tSolSat2A2} $\Sol$ is a sphere inside an affine
      subspace $\Asol$ of dimension $k := n - (m-2)$ that is
      perpendicular to $\Sat$. If $k = 1$, $\Sol$ consists of two
      points. The center of the sphere $\Sol$ is the point where
      $\Asol$ and $\Sat$ meet.
    \end{enumerate}
  \item \label{tSolSat2B} If all~$t_i$ are equal, then
    \begin{enumerate}[label=(\arabic*)]
    \item \label{tSolSat2B1} $\Sat$ is a sphere inside the affine
      subspace
      $\Asat = \langle\ve s_1 \upto \ve s_m\rangle_\mathrm{aff}$,
      which in this case has dimension $m - 1$. If $m = 2$, $\Sat$
      consists of the two points~$\ve s_1$ and~$\ve s_2$.
    \item \label{tSolSat2B2} $\Sol$ is an affine subspace of dimension
      $n - (m-1)$ that is perpendicular to $\Asat$. The center of the
      sphere $\Sat$ is the point where $\Sol$ and $\Asat$ meet.
    \end{enumerate}
  \end{enumerate}
  Notice that if neither the condition of \cref{tSolSat2A} nor of
  \cref{tSolSat2B} is met, we are in the situation of \cref{tSolSat}
  (or in the ``idiosyncratic'' situation of \cref{pCollinear}).
\end{theorem}

\begin{proof}
  We have $m \le n + 1$ since otherwise $A$ would have rank $n + 2$,
  contradicting the hypothesis that~\cref{eqGPSsquared} has at least
  two solutions.
  \begin{enumerate}
  \item[\ref{tSolSat2A}] Since $m - 1 \le n$, the hypothesis that
    $\ve s_1 \upto \ve s_m$ are not in general linear position means
    that $\langle\ve s_1 \upto \ve s_m\rangle_\mathrm{aff}$ has
    dimension $\le m - 2$, so $\rank(B) < m$. Hence the
    hypothesis~\cref{eqRank} is satisfied, and \cref{tSol2,tSat2}
    apply. Now \cref{tSol2} gives us \cref{tSolSat2A2} and, since
    $|\mathfrak X| \ge 2$, $\beta - \lVert\ve v\rVert^2 > 0$. So
    \cref{tSolSat2A1} follows from \cref{tSat2}.
  \item[\ref{tSolSat2B}] The hypothesis that all~$t_i$ are equal
    implies that the first and last column of $A$ are linearly
    dependent, so $m = \rank(A) = \rank(B)$, which with $m - 1 \le n$
    means that the $\ve s_i$ are in general linear position. Hence we
    are in the situation of \cref{exSphere1,exSphere2}. These yield
    \cref{tSolSat2B2,tSolSat2B1}, respectively.
    \endproof
  \end{enumerate}
\end{proof}

\cref{tSolSat2} is illustrated in \cref{fSolSat2}, which shows the case
$n = m = 3$.

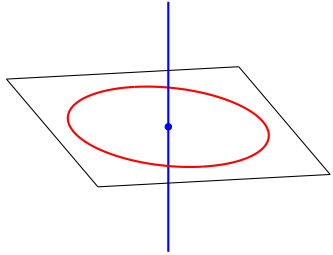
\begin{figure}[h]
  \vspace{-10mm}%
  \begin{tikzpicture}
    \begin{axis}[%
      view={78}{20},%
      xlabel = x,
      ylabel = y,
      zlabel = z,
      ymin=-5,
      ymax=5,
      samples=100,
      samples y=0,
      hide axis,%
      ]%
      
      \addplot3[%
      domain=-4.7:4.7,%
      color=white%
      ]%
      ({-5},{0},{x});%

      \addplot3[%
      domain=-5:5,%
      color=white%
      ]%
      ({x},{0},{-4.7});%
      
      \addplot3[%
      domain=-5:5,%
      color=white%
      ]%
      ({x},{0},{4.7});%
      
      \addplot3[%
      domain=-4.7:4.7,%
      color=white%
      ]%
      ({5},{0},{x});%
      
      \addplot3[%
      domain=-pi/2:pi/2,%
      color = red,
      thick,
      ]%
      ({3.7*sin(deg(x))},{2.2*cos(deg(x))},{0});

      \addplot3[%
      domain=-5:5,%
      ]%
      ({x},{-2.7},{0});%
      
      \addplot3[%
      domain=-5:5,%
      ]%
      ({x},{2.7},{0});%
      
      \addplot3[%
      domain=-2.7:2.7,%
      ]%
      ({-5},{x},{0});%
      
      \addplot3[%
      domain=-4.7:4.7,%
      color=blue,%
      thick%
      ]%
      ({0},{0},{x});%
      
      \addplot3[%
      domain=pi/2:3*pi/2,%
      color = red,
      thick,
      ]%
      ({3.7*sin(deg(x))},{2.2*cos(deg(x))},{0});

      \addplot3[%
      domain=-2.7:2.7,%
      ]%
      ({5},{x},{0});%

      \addplot3[%
      color=blue,%
      mark = *,
      mark size = 1.2,
      draw=none,
      ]%
      table[row sep=crcr]%
      {%
        0 0 0 \\%
      };%
    \end{axis}
  \end{tikzpicture}%
  \vspace{-12mm}
  \caption{An illustration of \cref{tSolSat2}, showing the set $\Sol$
    of solutions and the locus $\Sat$ of satellites (red and
    blue). Their roles are interchangeable.} \label{fSolSat2}
\end{figure}

A sphere has eccentricity~$0$, whereas an affine subspace can be
interpreted as having eccentricity~$\infty$. In this way, the results
of \cref{tSolSat2} fit the pattern from \cref{tSolSat}\ref{tSolSatD}
that the set of solutions and the locus of satellites have reciprocal
eccentricities.

\section{More generality and duality} \label{sDuality}

Until now we have assumed that the matrix 
$A \in \RR^{m \times (n+2)}$, as defined in~\cref{eqAGPS} and given by
the~$\ve s_i$ and~$t_i$, has rank~$m$. In particular this entails the
limitation $m \le n + 2$, which seems too restrictive. After all, in a
real-world use of global positioning one often has more than five
satellites in view. In this section we lift the rank hypothesis and we
will see that it is, when it comes to exact solutions (rather than
approximative ones), in fact a without-loss assumption. We start by
giving a procedure that gives the solution set of the squared
system~\cref{eqGPSsquared} without any rank hypothesis.

\begin{procedure}[Compute the set of solutions of the squared global
  positioning problem] \label{pGPS}%
  \mbox{}
  \begin{description}
  \item[Input] Points~$\ve s_1 \upto \ve s_r \in \RR^n$ and numbers~$t_1
    \upto t_r \in \RR$.
  \item[Output] The solution set $\mathfrak X$ of~\cref{eqGPSsquared},
    given as a quadric (which may be empty or otherwise degenerated).
  \end{description}
  \begin{enumerate}[label=(\arabic*)]
  \item \label{pGPS1} Form the matrices $A \in \RR^{r \times (n+2)}$
    and $B \in \RR^{r \times (n+1)}$ given by~\cref{eqAGPS,eqB}. Let
    $m := \rank(A)$ and reorder the~$\ve s_i$ and~$t_i$ (and with them
    the rows of $A$ and $B$) in such a way that the first~$m$ rows of
    $A$ are linearly independent. Then delete the last $r - m$ rows
    from $A$ and $B$, so now $A \in \RR^{m \times (n+2)}$ and
    $B \in \RR^{m \times (n+1)}$.
  \item \label{pGPS2} IF $\rank(B) = m$ THEN compute the vectors and
    scalars $\ve u,\ve v,\ve w_1 \upto \ve w_k \in \RR^n$ and
    $\alpha,\beta \in \RR$ as at the beginning of \cref{sQsol}. Let
    $\mathfrak X$ be defined as in \cref{tSol}, and let the function
    $\tilde h$ be given by
    \[
      \tilde h(t,\ve s) := \lVert\ve s\rVert^2 - \langle\ve u,\ve
      s\rangle^2 + 2 \langle\alpha \ve u - \ve v,\ve s\rangle + \beta
      - \alpha^2 \quad \text{for} \quad t \in \RR, \ \ve s \in \RR^n
    \]
    (which in this case does not depend on~$t$).
  \item \label{pGPS3} IF $\rank(B) < m$ THEN compute the vectors and
    scalars $\ve v,\ve w_1 \upto \ve w_k \in \RR^n$ and
    $b_0,\beta \in \RR$ as at the beginning of \cref{sRank}. Let
    $\mathfrak X$ be defined as in \cref{tSol2}, and let the function
    $\tilde h$ be given by
    \[
      \tilde h(t,\ve s) := (t - b_0)^2 - \lVert\ve s - \ve v\rVert^2 -
      \beta + \lVert\ve v\rVert^2 \quad \text{for} \quad t \in \RR, \
      \ve s \in \RR^n.
    \]
  \item \label{pGPS4} FOR $i:= m+1 \ldots r$ \quad DO
    \begin{enumerate}[label=(\arabic*)]
      \setcounter{enumii}{\value{enumi}}
    \item \label{pGPS5} IF $\tilde h(t_i,\ve s_i) \ne 0$ THEN set
      $\mathfrak X := \emptyset$.
    \end{enumerate}
    \setcounter{enumi}{\value{enumii}}
  \item \label{pGPS6} RETURN $\mathfrak X$.
  \end{enumerate}
\end{procedure}

The correctness of the algorithm follows from the following 
proposition, for which we need to introduce a bit of notation:

For a (possibly infinite) nonempty set $M \subseteq \RR \times \RR^n$
we write
\begin{equation} \label{eqXM}%
  \mathfrak X(M) := \Bigl\{(b,\ve x) \in \RR \times \RR^n \ \Bigl| \
  \lVert\ve s - \ve x\rVert^2 - (t - b)^2 = 0 \ \text{for all} \
  (t,\ve s) \in M\Bigr\}.
\end{equation}
There are pairs $(t_1,\ve s_1) \upto (t_m,\ve s_m) \in M$ such that
the matrix $A$ given by the~$\ve s_i$ and~$t_i$ has rank~$m$, and such
that this rank is the maximal possible. For these $(t_i,\ve s_i)$ we
also have the submatrix $B$ as defined in~\cref{eqB}. For the two
cases $\rank(B) = m$ and $\rank(B) < m$, the computation of the
solution set $\mathfrak X$ of~\cref{eqGPSsquared} was treated in
\cref{sQsol,sRank}. For this, the ``auxiliary'' vectors and scalars
$\ve u$, $\alpha$ (only if $\rank(A) = \rank(B)$), $\ve v$, $\beta$,
and $b_0$ (the latter only if $\rank(A) = \rank(B) + 1$) were formed.

\begin{prop} \label{pM}%
  In the above situation, for $(t,\ve s) \in \RR \times \RR^n$ set
  \[
    \tilde h(t,\ve s) :=
    \begin{cases}
      \lVert\ve s\rVert^2 - \langle\ve u,\ve s\rangle^2 + 2
      \langle\alpha \ve u - \ve v,\ve s\rangle + \beta -
      \alpha^2 & \text{if} \ \rank(B) = m \\
      (t - b_0)^2 - \lVert\ve s - \ve v\rVert^2 - \beta + \lVert\ve
      v\rVert^2 & \text{if} \ \rank(B) < m
    \end{cases}.
  \]
  If $\tilde h(t,\ve s) = 0$ for all
  $(t,\ve s) \in M \setminus \bigl\{(t_1,\ve s_1) \upto (t_m,\ve
  s_m)\bigr\}$, then $\mathfrak X(M) = \mathfrak X$. Otherwise
  $\mathfrak X(M) = \emptyset$.
\end{prop}

\begin{proof}
  Take
  $(t,\ve s) \in M\setminus \bigl\{(t_1,\ve s_1) \upto (t_m,\ve
  s_m)\bigr\}$, and form $A'$, $B'$, and $\mathfrak X'$ as in
  \cref{lSat,lSat2}. Because of the maximality of~$m$ we have
  $\rank(A') = m$, so the additional row of $A'$ is a linear
  combination of the first~$m$ rows. Therefore the same holds for the
  rows of $B'$, so $\rank(B') = \rank(B)$. Now \cref{lSat} or
  \cref{lSat2}, depending on the case, tell us that
  $\mathfrak X' = \mathfrak X$ if $\tilde h(\ve t,s) = 0$ and
  $\mathfrak X' = \emptyset$ otherwise. From this our claim follows.
\end{proof}

\begin{rem} \label{rM}%
  \cref{pM} also shows that if $\mathfrak X(M) \ne \emptyset$ then
  there is a finite subset $M' \subseteq M$ with $|M'| \le n + 2$ such
  that $\mathfrak X(M) = \mathfrak X(M')$. If
  $\mathfrak X(M) = \emptyset$ the upper bound is $|M'| \le n + 3$. Of
  course $\mathfrak X(M) = \emptyset$ can also be achieved with
  $|M| = 3$: just set
  $M = \bigl\{(t_1,\ve s),(t_2,\ve s),(t_3,\ve s)\bigr\}$ with
  pairwise distinct~$t_i$.
\end{rem}

\begin{rem} \label{rDuality}%
  To push the idea of duality further, it is convenient to change the
  notation by writing $M^* := \mathfrak X(M)$ for the set defined
  by~\cref{eqXM}, so $M^*$ is the set of solutions
  of~\cref{eqGPSsquared}, but with potentially infinitely many
  equations. Since there is complete symmetry between $(b,\ve x)$ and
  $(t,\ve s)$ in~\cref{eqXM}, we have
  \[
    M \subseteq M^{**}.
  \]
  For a subset $N \subseteq M$ we clearly have $M^* \subseteq
  N^*$. These two rules imply $M^{***} = M^*$. Passing from a set
  $M \subseteq \RR \times \RR^n$ to $M^{**}$ can be seen as a closure
  operation, and we conclude that the closed sets are precisely those
  of the form $M^*$. By \cref{rM}, $M$ can be chosen finite of
  size $\le n + 2$. To reconnect this notation with the one used
  during the bulk of this paper, notice that in the special case
  $M = \bigl\{(t_1,\ve s_1) \upto (t_m,\ve s_m)\bigr\}$, the sets
  $\mathfrak X$ and $\mathcal S$ defined in \cref{eqX,eqS} are
  $\mathfrak X = M^*$ and $\mathcal S = \mathfrak X^* = M^{**}$, so
  $\mathfrak X = \mathcal S^*$.

  For ``closed'' subsets of $\RR \times \RR^n$ we thus have perfect
  duality. If we project them to $\RR^n$, we obtain the sets $\Sol$
  and $\Sat$, which are geometrically more meaningful and which also
  display some duality that is perhaps best illustrated by
  \cref{fQuadrics,fSolSat2}. But this duality breaks down ``on the
  fringes'' due to a loss of information from the projection. For
  instance, if $\Sol$ is a line, $\Sat$ may be another line in the
  situation of \cref{pCollinear}, or a circle in the situation of
  \cref{fSolSat2}. The results of this paper can now be expressed by
  saying that the projections to $\RR^n$ of closed subsets $M$ are
  spheres, prolate spheroids, hyperboloids of revolution of two
  sheets, pairs of points, affine subspaces, single points, or
  empty. They also say that duality exchanges foci and vertices,
  inverts the eccentricity, but preserves the axis of symmetry.
\end{rem}

\section{Two application examples}
\label{sec:examples_applications}
\subsection{Indoor vacuum robot}
\label{sec:vacuum_robot}

Consider a robot vacuum rolling on the floor of an apartment.
The robot is using emitters positioned on the ceiling of the rooms of the apartment to determine its location.
The emitters are equipped with precisely synchronized clocks and their locations are known precisely. 
They are continuously emitting their location and time of emission.

Assume that the robot is situated in a room with $m=3$ non-collinear emitters. Assume also that the robot can receive the signals of all three emitters from any position on the ground of the room. Without loss of generality, we can assume that the ground floor is in the linear space spanned by the standard basis vectors $\ve e_1$ and $\ve e _2$. Then $U=\langle\ve s_3-\ve s_1,\ve s_2-\ve s_1\rangle$ is the linear space perpendicular to $\ve e_3$ and so $U^T=\RR \cdot \ve e_3$.

Since the three emitters are not collinear, there exists a unique circle on the ceiling plane that passes through the three emitters' positions. 
If the robot is positioned directly below the center of that circle, we are in the situation of \cref{exSphere1}. So the three pseudoranges received by the robots are equal, $\ve v$ is the center of the circle and $\ve u=\ve 0$. Then the set of possible solutions for the location of the robot is
\[ \Sol = \{ \ve v+y_1 \ve e_3 | y_1\in \RR \} .\] 
In other words, it
consists of a line perpendicular to the ceiling and passing through both the true position of the robot and the center of the circle. From the pseudoranges alone, it would thus be impossible to determine the height of the robot in the room. But by intersecting the line $\Sol$ with the floor, a unique location can be determined.
Now the set $\Sat$ corresponds to the circle through the three emitters. Therefore, positioning an additional emitter on that circle would not provide any further information regarding the robot position. In particular, if the original three emitters were in the corners of a rectangular  ceiling, an additional emitter positioned on the remaining corner would only provide redundant information. However, more information could be obtained by moving the additional emitter away from the circle. In particular, 
if the robot's position was not constrained to the floor plane, 
then positioning an additional emitter on the ceiling while avoiding the circle through the original three emitters would 
reduce the solution set to two mirror points above and below the ceiling, respectively.

Now, let us consider the situation where the robot on the ground is not situated directly below the center of the circle defined by the emitters. This is illustrated in \cref{fIneq}. Then we are in the situation discussed in \cref{sShape}. In that case, the solution space $\Sol$ for the robot position is contained in the 2D affine plane $\{ v+y_0 u+y_1 e_3 | y_0,y_1\in \RR \}$ and is either an ellipse (2D spheroid), a hyperbola (2D hyperboloid), or a parabola (2D paraboloid). Using the fact that the robot is on the ground, one can intersect the corresponding quadric with the ground plane to better pin point the robot's position. Intersecting the floor plane with a parabola ($e=1$) will yield a single solution; intersecting it with either an ellipse or a hyperbola ($e\neq 1$) will yield two possible solutions. One can use the condition $t_1,t_2,t_3\geq t$ to try to rule out the extraneous solution. Recall that, when $\Sol$ is an ellipse, then the set $\Sat$ is a hyperbola and vice versa. If $\Sat$ is an ellipse, then applying the condition $t_1,t_2,t_3\geq t$ is guaranteed to yield a single possible solution for the robot position. However, when  $\Sat$ is a hyperbola, then as we have proven (see \cref{fIneq}) that the emitters must be on one sheet of that hyperbola. This implies that there will be two solutions for the robot position. 
Thus, the localization problem will have a unique solution on the ground in all cases except when $\Sat$ is a hyperbola.

\begin{figure}
    \centering
\begin{tikzpicture}[scale=1]


\coordinate (A) at (-2,-1.5);
\coordinate (B) at (2,-1.5);
\coordinate (C) at (2,1.5);
\coordinate (D) at (-2,1.5);

\coordinate (A') at (-1.5,-1);
\coordinate (B') at (2.5,-1);
\coordinate (C') at (2.5,2);
\coordinate (D') at (-1.5,2);

\coordinate (DE) at (-1.8,1.6);
\coordinate (CE') at (2.05,1.7);
\coordinate (CE) at (-1.1,2);

\draw (A)--(B)--(C)--(D)--cycle;
\draw (A')--(B')--(C')--(D')--cycle;
\draw (A)--(A'); \draw (B)--(B'); \draw (C)--(C'); \draw (D)--(D');



\fill[blue] (CE) circle (0.04);
\node at ($(CE)-(0,0.25)$) {\small Emitter};
\draw[blue] (CE) circle (0.08);
\draw[blue] (CE) circle (0.12);
\draw[blue] (CE) circle (0.16);

\fill[blue] (DE) circle (0.04);
\node at ($(DE)-(0,0.25)$) {\small Emitter};
\draw[blue] (DE) circle (0.08);
\draw[blue] (DE) circle (0.12);
\draw[blue] (DE) circle (0.16);

\fill[blue] (CE') circle (0.04);
\node at ($(CE')-(0,0.25)$) {\small Emitter};
\draw[blue] (CE') circle (0.08);
\draw[blue] (CE') circle (0.12);
\draw[blue] (CE') circle (0.16);


\coordinate (TopCenter) at ($(D)+0.5*(C')-0.5*(D)$);

\coordinate (BottomCenter) at ($(A)+0.5*(B')-0.5*(A)$);

\draw[dotted, thick] (TopCenter) -- (BottomCenter);



\fill[red] (BottomCenter) ellipse [x radius=0.2, y radius=0.05];
\node at ($(BottomCenter)+(0,-0.15)$) {\footnotesize Robot Vacuum};

        \draw[dotted,red][rotate=8] (TopCenter) ellipse [x radius=2.1, y radius=0.555];

\end{tikzpicture}
    \caption{When locating a robot vacuum on the ground using three ceiling-mounted emitters, there are at most two solutions. When the robot is directly below the center of the circle passing through the emitters, the solution is unique. }
    \label{fig:enter-label}
\end{figure}
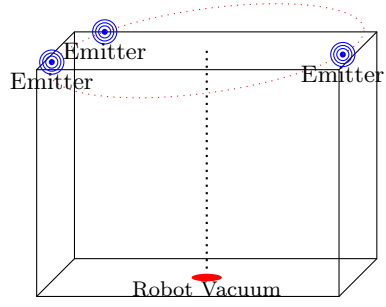


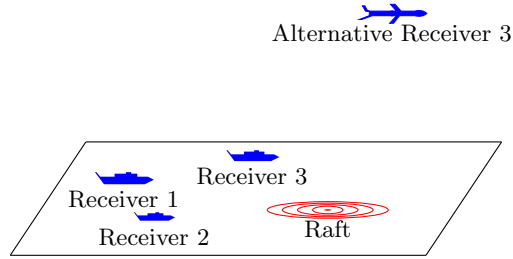
\begin{figure}
    \centering
\begin{tikzpicture}[scale=1]
\coordinate (L) at (-3,-2);
\coordinate (R) at (2.5,-2);
\coordinate (L') at (-2,-0.5);
\coordinate (R') at (3.5,-0.5);

\coordinate (DE) at (-1.5,-1); %
\coordinate (CE') at (2.05,1.2); 
\coordinate (CE) at (-1.1,-1.5); 
\coordinate (BOAT) at (0.2,-0.7); 
\coordinate (RAFT) at (1.2,-1.4); 

\draw (L)--(R);
\draw (L')--(R');
\draw (L)--(L'); \draw (R)--(R');

\fill[blue] (BOAT) circle (0.04);
\node at ($(BOAT)-(0,0.25)$) {\small Receiver 3};
\begin{scope}[shift={(BOAT)}, scale=0.09]
\fill[blue]
  (4,0)
  -- (3.2,0.4)
  -- (1.8,0.4)
  -- (1.8,0.8)
-- (1.0,0.8)
-- (1.0,1.0)
-- (0.3,1.0)
-- (0.3,0.8)
  -- (-1.2,0.8)
  -- (-1.2,0.4)
  -- (-2.5,0.4)
  -- (-3.1,0.4)
  -- (-3.6,1.0)
  -- (-3.8,1.0)
  -- (-3.2,0.2)
  -- (-2.6,-0.6)
  -- (2.8,-0.6)
  -- cycle;
\end{scope}

\node at ($(CE)-(0,0.25)$) {\small Receiver 2};
\begin{scope}[shift={(CE)}, scale=0.07]
\fill[blue]
  (4,0)
  -- (3.2,0.4)
  -- (1.8,0.4)
  -- (1.8,0.8)
-- (1.0,0.8)
-- (1.0,1.0)
-- (0.3,1.0)
-- (0.3,0.8)
  -- (-1.2,0.8)
  -- (-1.2,0.4)
  -- (-2.5,0.4)
  -- (-3.1,0.4)
  -- (-3.6,1.0)
  -- (-3.8,1.0)
  -- (-3.2,0.2)
  -- (-2.6,-0.6)
  -- (2.8,-0.6)
  -- cycle;
\end{scope}

\node at ($(DE)-(0,0.25)$) {\small Receiver 1};
  \begin{scope}[shift={(DE)}, scale=0.1]
\fill[blue]
  (4,0)
  -- (3.2,0.4)
  -- (1.8,0.4)
  -- (1.8,0.8)
-- (1.0,0.8)
-- (1.0,1.0)
-- (0.3,1.0)
-- (0.3,0.8)
  -- (-1.2,0.8)
  -- (-1.2,0.4)
  -- (-2.5,0.4)
  -- (-3.1,0.4)
  -- (-3.6,1.0)
  -- (-3.8,1.0)
  -- (-3.2,0.2)
  -- (-2.6,-0.6)
  -- (2.8,-0.6)
  -- cycle;
  \end{scope}

\fill[blue] (CE') circle (0.04); 
\node at ($(CE')-(0,0.25)$) {\small Alternative Receiver 3};
\begin{scope}[shift={(CE')}, scale=0.1]

\filldraw[blue]
  (4.6,0)
  .. controls (4.2,0.4) and (3.6,0.45) .. (3,0.3)
  -- (1.2,0.3)
  -- (0.4,1.2)
  -- (-0.2,1.2)
  -- (0.6,0.3)
  -- (-3.2,0.3)
  -- (-3.6,1.0)
  -- (-3.9,1.0)
  -- (-3.5,0.3)
  -- (-3.0,0.3)
  -- (-3.0,-0.3)
  -- (-4.3,-0.8)
  -- (-4.0,-0.8)
  -- (-2.6,-0.3)
  -- (0.6,-0.3)
  -- (-0.3,-1.3)
  -- (0.2,-1.3)
  -- (1.0,-0.3)
  -- (3,-0.3)
  .. controls (3.6,-0.45) and (4.2,-0.4) .. (4.6,0)
  -- cycle;
\end{scope}

\fill[red] (RAFT) ellipse [x radius=0.05, y radius=0.01];
\draw[red] (RAFT) ellipse [x radius=0.2, y radius=0.05];
\draw[red] (RAFT) ellipse [x radius=0.4, y radius=0.07];
\draw[red] (RAFT) ellipse [x radius=0.6, y radius=0.09];
\draw[red] (RAFT) ellipse [x radius=0.8, y radius=0.11];
\node at ($(RAFT)-(0,0.25)$) {\small Raft};
\end{tikzpicture}
    \caption{Three boats on the sea can accurate determine the two possible positions of a raft emitting signal because the solution quadric of solutions intersect the sea plan perpendicularly. The conditioning of the location problem is significantly worse if one of the receivers is placed high up in the air (e.g., on an airplane) instead of on the sea. }
    \label{fig:enter-label}
\end{figure}


\subsection{Locating a raft on the ocean}
\label{sec:life_boat}
Imagine a life raft drifting on the ocean. In order to be found, the people on the raft are emitting signals. These could be short bursts of sound, light flashes or electromagnetic signals from some device. If the signal is received by three non-collinear receivers, and if the clocks on the receivers are precisely synchronized, then the method discussed in this paper can be used to try to locate the raft. 

Let us explore a first scenario in which there are three boats nearby and each of the boat receives the distress signal.  Let us assume that the ocean surface is flat. In that case, the search space for the raft is a plane (the ocean) and the receivers can be assumed to be situated on the same plane.
While this examples bears a lot of resemblance to the  previous example of the cleaning robot, the fact that the solution space and the receiver plane coincide makes a significant difference: now the quadric of solution passes through the ocean (receiver plane) {\em perpendicularly}. 
The solution(s) being at the intersection point between the solution curve and the receiver plane, this should make the numerical solution robust to noise, as the intersection point should remain more or less the same when small errors are introduced in the time-of-arrival measurements and the boats move up and down with the waves.

This is illustrated in \cref{fig:3boats} where we investigate the situation where three boats within 1-2 kilometers of each other on the ocean plane receive the signal from a nearby raft on that same plane. For that experiment, we set the boat positions at $\ve s_1=(0,0,0)$, $\ve s_2=(1,1,0)$ and $\ve s_3=(0,1,0)$ respectively, and we placed the raft at position $\ve x =(0.75, 5, 0)$. Setting the clock bias at $b=0$, we then computed the times-of-arrival $t_i=\| \ve s_i - \ve x \|$ for each boat, and perturbed each of them independently with additive Gaussian noise with $\sigma=0.01$ in order to obtained noisy times-of-arrival $\tilde{t}_i$. In order to visualize the numerical behavior of the solution, we considered every possible location on a fine grid within a search rectangle on the ocean plane. For each location $\ve p$, we obtained the corresponding distances $\| \ve s_i - \ve p \|$ to the boats, and we used those to obtain an estimate $b_{\ve p}$ the clock bias $b$. More specifically we set $b_{\ve p}=\frac{1}{3}\sum_{i=1}^3(\| \ve s_i - \ve p \|-\tilde{t}_i)$ and $t_{i,\ve p}=\| \ve s_i - \ve p \|+b_{\ve p}$. The surface shown in the figure represents the value $E$ of the mean-squared-error of the reconstructed times-of-arrival, namely \[ E=\sum_{i=1}^3 (t_{i,\ve p}- \tilde{t}_i)^2.\] As one can see from the plot, there is a clearly defined (very ``pointy") global minimum on the grid of the search area. Let us take the grid point corresponding to the minimum value to be the approximate solution to the localization problem. \cref{tab:error_ocean} summarizes the accuracy of the results (over $10,000$ trials)of this simple solution method. B Average results for both the aforementioned case of $\sigma=0.01$ as well as for the case where $\sigma=0.1$ are listed. As one would expect,  when the noise is smaller, the solution obtained is more accurate, on average, and the error tends to vary less.

In \cref{fig:2boats_1plane}, we contrast this situation with that where one of the receiver is on an aircraft, rather than on a boat. For that experiment, we set the boat positions at $\ve s_1=(0,0,0)$, $\ve s_2=(1,1,0)$, the aircraft position at $\ve s_3=(0,1,10)$, and we kept the raft at the same position $\ve x =(0.75, 5, 0)$. Again, we set the clock bias at $b=0$, computed the times-of-arrival $t_i=\| \ve s_i - \ve x \|$ for each boat, and perturbed each of them independently with additive Gaussian noise with $\sigma=0.01$. The cost function $E$, computed as in the previous case, is shown in \cref{fig:2boats_1plane}. Notice how the local minimum in this case is less ``pointy" than when the three receivers are on the ocean plane. One can see that it lies within a flat valley within the search space. By setting the solution at the global minimum of the cost function $E$, we can again obtain an approximate solution to the localization problem. The experimental average and standard deviation of the error of this solution (over $10,000$ trials) are shown in \cref{tab:error_ocean}. As one can see, the solution obtained is less accurate, on average, and tends to vary more when the third receiver is located far up in the sky rather than on the ocean plane. This illustrates the numerical advantage of setting the receivers approximately on the same plane as the solution space.

\begin{table}[h]
    \centering
    \begin{tabular}{l|llll}
          &  $1\%$ noise &  &  $10\%$  noise & \\
          \hline
          & mean & standard deviation & mean &standard deviation\\
         \hline 
         Three boats &  0.31m & 1.39m  & 16.84m  & 12.55m \\
       Two boats and one aircraft &  12.68m & 11.40m  & 26.02m  & 17.21m \\
    \end{tabular}
    \caption{Position error for the case of receivers on three boats (solution quadric perpendicular to search space) and the case of receivers on two boats and one aircraft (solution quadric not perpendicular to search space). Here the solution is obtained by taking the grid point on the portion of the plane being searched where the cost function is the smallest. We see that the solution is more accurate and has lower standard deviation when the solution quadric is perpendicular to the search space.}
    \label{tab:error_ocean}
\end{table}



\begin{figure}[h]
    \centering
    \includegraphics[width=0.45\linewidth]{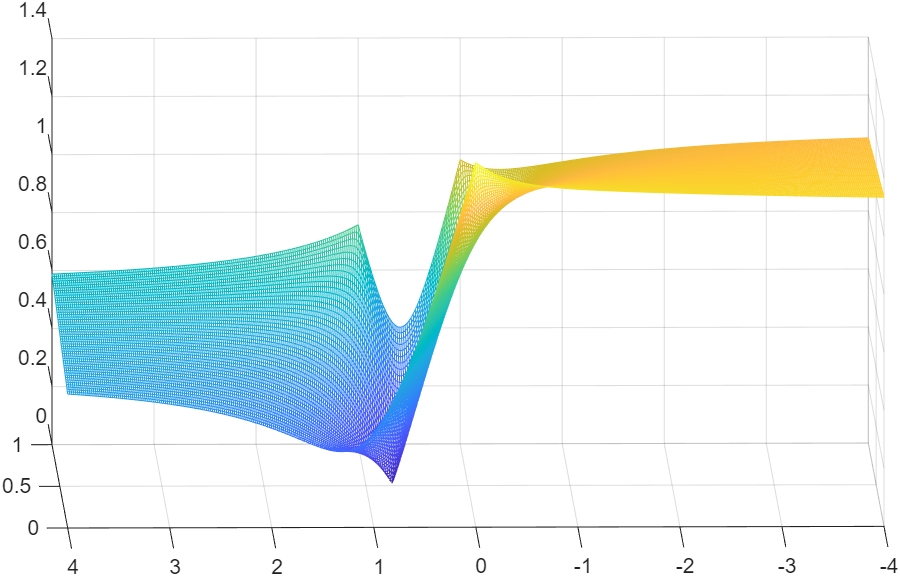}
    \includegraphics[width=0.45\linewidth]{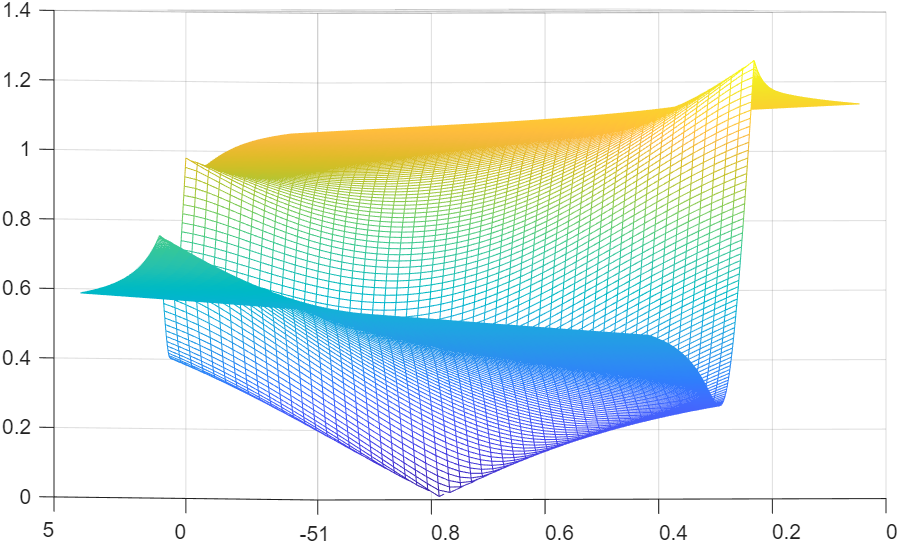}
    \caption{{\bf Locating a life raft emitting a signal received by three boats:} Two views of the mean-squared-error for all locations  on grid inside search rectangle on ocean plane, with $1\%$ noise on the signal times-of-arrival. The boats are around $1$-$1.5$ km from each other and the life raft is within $1$ km from each boat. The perpendicularity of the solution quadric and the search plane create a well-defined global minimum within the search area, thus enabling precise location of the raft.}
    \label{fig:3boats}
\end{figure}


\begin{figure}[h]
    \centering
    \includegraphics[width=0.45\linewidth]{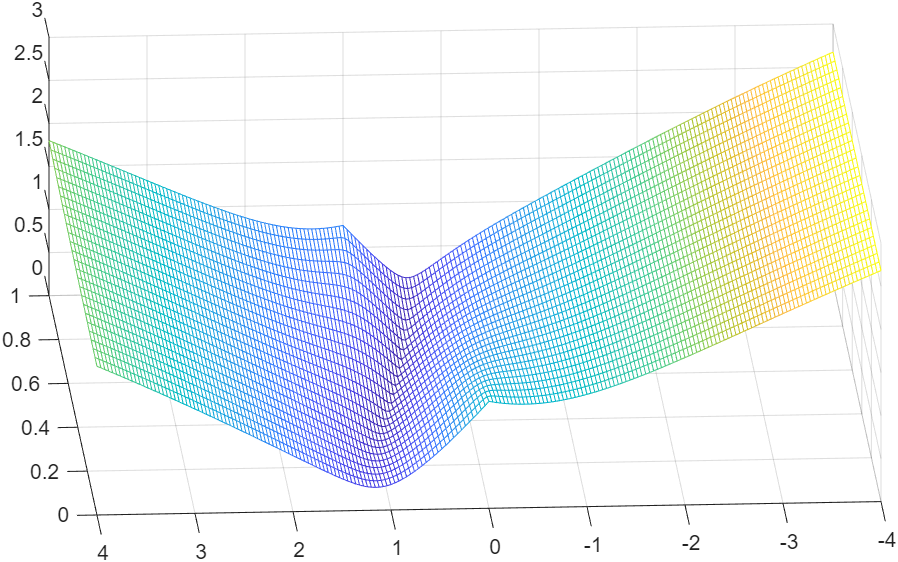}
    \includegraphics[width=0.45\linewidth]{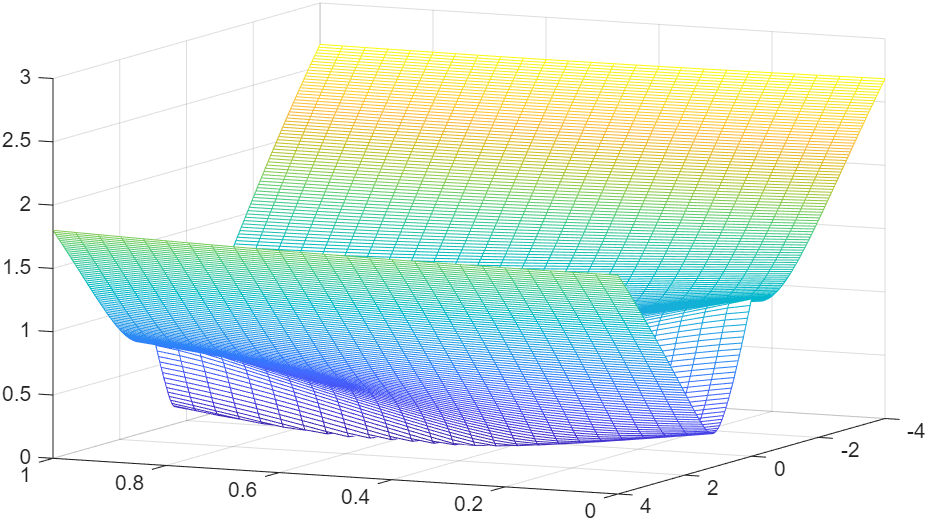}
    \caption{{\bf Locating a life raft emitting a signal received by two boats and an aircraft:} Two views of the mean-squared-error for all locations on grid inside search rectangle, with $1\%$ noise on the signal times-of-arrival. The two boats are about $1.5$ km from each other, the life raft is within $1$ km from each boat, and the aircraft is flying $10$ km high in the sky above. Now the solution quadric is not perpendicular to the search space, and the global minimum within the search area is in a long narrow valley, making it harder to locate the raft.}
    \label{fig:2boats_1plane}
\end{figure}

\end{document}